\documentclass[12pt, a4paper,reqno]{amsart}
\usepackage{amscd,amssymb,amsmath, array,latexsym,amsbsy,color}

\allowdisplaybreaks[2]
\begin{document}


\def\Stab{\textrm{Stab}}
\def\supp{\operatorname{supp}}
\def\tr{\operatorname{tr}}
\def\lt{{\operatorname{{lt}}}}
\def\Ind{\operatorname{Ind}} 
\def\Res{\operatorname{Res}}
\def\Aut{\operatorname{Aut}}
\def\Prim{\operatorname{Prim}}
\def\Rep{\operatorname{Rep}}
\def\Id{\operatorname{Id}}
\def\GFA{\operatorname{FA}}
\def\characters{\operatorname{char}}
\def\SO{\operatorname{SO}}
\def\sign{\operatorname{sign}}
\def\dashind{\operatorname{\!-Ind}}
\def\lsp{\operatorname{span}}
\newcommand{\clsp}{\overline{\lsp}}

\def\H{\mathcal{H}} 
\def\K{\mathcal{K}} 
\def\N{\mathcal{N}} 
\def\C{\mathbb{C}}
\def\T{\mathbb{T}}
\def\Z{\mathbb{Z}}
\def\R{\mathbb{R}}
\def\P{\mathbb{P}}
\def\NN{\mathbb{N}}

\newtheorem{thm}{Theorem}[section]
\newtheorem{cor}[thm]{Corollary}
\newtheorem*{claimA*}{Claim A}
\newtheorem*{claimB*}{Claim B}
\newtheorem*{claimC*}{Claim C}
\newtheorem{claim}[thm]{Claim}
\newtheorem{prop}[thm]{Proposition}
\newtheorem{lemma}[thm]{Lemma}
\theoremstyle{definition}
\newtheorem{defn}[thm]{Definition}
\newtheorem{remark}[thm]{Remark}
\newtheorem{example}[thm]{Example}
\newtheorem{examples}[thm]{Examples}
\numberwithin{equation}{section}
\title
[\boldmath The $C^*$-algebras of compact transformation groups]
{\boldmath The $C^*$-algebras of compact transformation groups}

\author[Archbold]{Robert Archbold}
\address{Institute of Mathematics
\\University of Aberdeen
\\Aberdeen AB24 3UE
\\Scotland
\\United Kingdom
}
\email{r.archbold@abdn.ac.uk}

\author[an Huef]{Astrid an Huef}
\address{Department of Mathematics and Statistics\\
University of Otago\\
Dunedin 9054\\
New Zealand}
\email{astrid@maths.otago.ac.nz}

\date{Submitted 6 January 2014, revised 14 October 2014}

\begin{abstract} We investigate the representation theory  of the crossed-product $C^*$-algebra  associated to a compact group $G$ acting on a locally compact space $X$ when the stability subgroups vary discontinuously.
Our main result applies when $G$ has a principal stability subgroup or $X$ is locally of finite $G$-orbit type. Then the  upper multiplicity of the  representation of the crossed product induced from an irreducible representation $V$ of a stability subgroup is obtained by restricting $V$ to a certain closed subgroup of the stability subgroup and taking the maximum of the multiplicities of the irreducible summands occurring in the restriction of $V$. As a corollary we obtain that when the trivial subgroup is a principal stability subgroup,  the crossed product is a Fell algebra if and only if every stability subgroup is abelian. A second corollary is that the $C^*$-algebra of the  motion group $\R^n\rtimes \SO(n)$ is a Fell algebra. This uses the classical branching theorem for the special orthogonal group $\SO(n)$ with respect to $\SO(n-1)$. Since proper transformation groups are locally induced from the actions of compact groups, we describe how some of our results can be extended to transformation groups that are locally proper.
\end{abstract}

\thanks{This research was supported by grant number 41207 from the London Mathematical Society and by travel grants from the Universities of Aberdeen and Otago. Part of the research was undertaken during the workshop ``The structure and classification of nuclear $C^*$-algebras'' at the International Centre  for Mathematical Sciences, Edinburgh, 2013.}

\maketitle

\section{Introduction} It is well-known that a $C^*$-algebra $A$ has continuous trace if and only if (1) the spectrum  $\hat A$ of $A$  is Hausdorff and (2) every $\pi\in\hat A$ satisfies Fell's condition: there exist a neighbourhood $O$ of $\pi$ in $\hat A$ and a positive element $a$ of $A$ such that $\rho(a)$ is a rank-one projection for all $\rho\in O$ \cite[Propositions~4.5.3 and 4.5.4]{Dix2}.
$C^*$-algebras $A$ for which (2) holds are called Fell algebras.
The Fell algebras are precisely those $C^*$-algebras which are generated by their abelian elements (the $C^*$-algebras  of Type I${}_0$  \cite[\S 6.1]{Ped}). Since the spectrum of a Fell algebra is locally Hausdorff by \cite[Corollary~3.2]{ASo}, one may think of a Fell algebra as having continuous trace locally in the sense that there is a collection of ideals of $A$ with continuous trace  that together generate $A$ \cite[Theorem~3.3]{aHKS}.

Let $G$ be a compact group acting continuously on a locally compact space $X$ and consider the associated crossed-product $C^*$-algebra  $C_0(X)\rtimes G$.  Echterhoff showed in \cite[Corollary~2]{E} that $C_0(X)\rtimes  G$ has continuous trace if and only if its spectrum is Hausdorff.  A comparison with the characterisation of $C^*$-algebras with continuous trace above raises the question of whether $C_0(X)\rtimes  G$ is always a Fell algebra, and indeed this is the case when the compact group $G$ is also abelian \cite[Lemma 5.10]{aH}.  However, \cite[Remark 2.6]{AK} shows that the answer can be negative even for the finite permutation group $G=S_3$ (see also Example~\ref{ex-S3} below).

The purpose of this paper is therefore to investigate the extent to which irreducible representations of $C_0(X)\rtimes G$ may fail to satisfy Fell's condition. There are several guiding principles for this. The first is that the failure of Fell's condition implies that the stability subgroups of $G$ vary in a discontinuous manner \cite[Corollary~2]{E}.
Secondly,  if $x_n\to z$ in $X$ then it follows from the sequential compactness of the space $\Sigma(G)$ of closed subgroups of $G$ that there is a closed subgroup $H$ of $G$ and a subsequence $(x_{n_k})_{k}$ such that $S_{x_{n_k}}\to H$. Moreover, since $G$ acts continuously on the Hausdorff space $X$, $H\subset S_z$.
Thirdly, the extent to which an irreducible representation $\pi$ of a $C^*$-algebra $A$ may fail to satisfy Fell's condition is measured by the upper multiplicity $M_U(\pi)$: $\pi$ satisfies Fell's condition if and only if the upper multiplicity $M_U(\pi)$ takes the value $1$ \cite[Theorem 4.6]{A}. Finally,  proper transformation groups are locally induced  from compact transformation groups \cite{abel}. Thus we begin by investigating the upper multiplicities of the irreducible representations of $C_0(X)\rtimes G$ when $G$ is compact. In \S\ref{sec-Cartan}  we  extend some of our results to the more general setting of  transformation groups that are locally proper.

Let $G$ be a compact group acting on a space $X$. This action is then integrable in the sense of \cite{aH, rieffel2} and so $C_0(X)\rtimes G$ is a $C^*$-algebra with bounded trace by \cite[Proposition 3.5]{aH}. It follows from \cite[Theorem 2.6]{ASS} that the upper multiplicities of the irreducible representations of $C_0(X)\rtimes G$ are all finite and therefore take values in the positive integers.
 Since $G$ is compact the orbits are closed in $X$, and it follows that every irreducible representation of $C_0(X)\rtimes G$ is unitarily equivalent to one that is induced from an irreducible representation $V$ of some stability subgroup $S_z:=\{g\in G: g\cdot z=z\}$ (see \S\ref{subsec-inducedreps}  below). We write $\Ind_{S_z}^G(\pi_z\rtimes V)$ for such a representation.
 Our first main result for compact group actions, Theorem~\ref{thm-new-general}, states that for $z\in X$ and $V\in{\hat S}_z$, there exist a closed subgroup $H$ of $G$ with $H\subset S_z$ and $R\in \hat H$
such that
\[M_U\big(\Ind_{S_z}^G(\pi_z\rtimes V)\big)=\big[ V|_H:R \big]\leq \dim V,\]
where $[V|_H:R]$ is the multiplicity of $R$ in $V|_H$.

Our second main result, Theorem~\ref{thm-pss}, gives more detailed information in the following two situations: when $G$ has a principal stability subgroup or when $X$ is locally of finite $G$-orbit type.
    A closed subgroup $S$ of $G$ is  a \emph{principal stability group}\label{page-situations} if the $G$-invariant subset \[\{x\in X: S_x \mbox{ is conjugate to } S\}\] is dense in $X$.  The space $X$ is said to be \emph{of finite $G$-orbit type} if there exist $n\in \NN$ and closed subgroups $H_1,\ldots,H_n$ of $G$ such that for each $x\in X$, $S_x$ is conjugate to one of the $H_i$. Then $X$ is said to be \emph{locally of finite $G$-orbit type} if every point in $X$ has a $G$-invariant neighbourhood which is of finite $G$-orbit type.

   If $G$ is a compact Lie group and $X$ is a topological manifold, then $X$ is locally of finite $G$-orbit type \cite[Remark after IV.1.2]{Bredon}, and if in addition $X/G$ is connected then there exists a principal stability group \cite{Mont}. It is well-known that if $G$ is compact and $S$ is a principal stability group then, for each $x\in X$, $S_x$ contains a conjugate of $S$ (see, for example, the argument in Claim~B of the proof of Theorem~\ref{thm-pss} below). It follows from \cite[Proposition~1.9]{Bredon} that if a principal stability group exists then it is unique up to conjugacy (cf. \cite[Remark 4.1]{KST}).

Theorem~\ref{thm-pss} says that if $G$ contains a principal stability subgroup or if $X$ is locally of finite $G$-orbit type then, for $z\in X$ and $V\in{\hat S}_z$, there exists a closed subgroup $H$ of $G$ with $H\subset S_z$
such that
\[M_U\big(\Ind_{S_z}^G(\pi_z\rtimes V)\big)=\max_{R\in\hat H}\big[ V|_H:R \big]\leq \dim V.\]
Moreover, $H$ may be chosen to be conjugate to a principal stability subgroup if such exists.

Theorem~\ref{thm-pss} leads to a number of corollaries.  For example, when the trivial subgroup $\{e\}$ is a principal stability subgroup, then $C_0(X)\rtimes G$ is a Fell algebra if and only if every stability subgroup is abelian (Corollary~\ref{new-cor-pss-trivial}). The $C^*$-algebra of the classical motion group $\R^n\rtimes \SO(n)$ is isomorphic to $C_0(\R^n)\rtimes\SO(n)$, and it follows from Theorem~\ref{thm-pss} and the branching theorem for the special orthogonal group $\SO(n)$ with respect to the principal stability subgroup $\SO(n-1)$ that $C_0(\R^n)\rtimes\SO(n)$ is a Fell algebra (Example~\ref{ex-orthogonal}).

In \S\ref{sec-Fellpoints}, we study irreducible representations $\Ind_{S_z}^G(\pi_z\rtimes V)$ of $C_0(X)\rtimes G$ where $z\in X$ and $V$ is the restriction to $S_z$ of a character of the compact group $G$. We show in Corollary~\ref{cor-open-Fell} that the collection of all such irreducible representations of $C_0(X)\rtimes G$ is an open subset of $(C_0(X)\rtimes G)^\wedge$ and that the corresponding ideal of $C_0(X)\rtimes G$ is a Fell algebra.

The Cartan transformation groups of \cite{palais} are locally proper, and proper transformation groups are locally induced from  compact transformation groups  \cite{abel}.  In \S\ref{sec-Cartan} we discuss how this can be used to extend some of our results from  compact to Cartan transformation groups. In particular,  Theorem~\ref{thm-Cartan} extends Theorem~\ref{thm-new-general}.

We have been motivated by a number of ideas in \cite{Bag, E, EH}, although we have not needed to use the slice property of Palais \cite[Definition 1.7]{EH}. We are grateful to the referee for drawing our attention to \cite{abel} and \cite{Neu}, and for suggesting strategies which have led to  Theorems~\ref{thm-new-general} and \ref{thm-Cartan}.

\section{Preliminaries}\label{sec-prelim}
Throughout $X$ is a second-countable, locally compact, Hausdorff space and  $G$ is a second-countable group with a jointly continuous action $(g,x)\mapsto g\cdot x$ of $G$ on $X$. In   \S\S\ref{sec-prelim}-\ref{sec-Fellpoints} the group $G$ is compact, and we summarise this set-up as ``$(G,X)$ is a second-countable transformation group with $G$ compact''.

 Let $C_0(X)$ be the $C^*$-algebra of continuous functions $f:X\to \C$ vanishing at infinity. Then $G$ acts on $C_0(X)$  by left translation via $\lt_s(f)(x)=f(s^{-1}\cdot x)$. The action $\lt$ of $G$ on $C_0(X)$ restricts to an action of any subgroup of $G$ on $C_0(X)$, and we denote these restrictions by $\lt$ as well.

A pair $(\pi,U)$, where $\pi:C_0(X)\to B(\H)$ is a representation on a Hilbert space $\H$ and $U:G\to UB(\H)$ is a representation of $G$ by unitary operators on $\H$, is called covariant for $(C_0(X), G,\lt)$ if
$U_s\pi(f)U_s^*=\pi(\lt_s(f))$ for all $f\in C_0(X)$ and $s\in G$.
The crossed product $C^*$-algebra $C_0(X)\rtimes_\lt G$ (or simply $C_0(X)\rtimes G$) is the $C^*$-algebra which is universal for covariant representation  of $(C_0(X), G,\lt)$, see \cite{tfb^2}.

Since both $G$ and $X$ are second-countable, the group $C^*$-algebra $C^*(G)$ and the $C^*$-algebras $C_0(X)$ and $C_0(X)\rtimes_\lt G$ are all separable.  Hence all of their irreducible representations act on separable Hilbert spaces.

We write $\hat G$ for the (equivalence classes of) irreducible unitary representations of $G$.  We often identify the unitary representations of $G$ and the non-degenerate representations of $C^*(G)$ in the usual way.

By an ideal of a $C^*$-algebra $A$ we  always mean a norm-closed two-sided ideal. Let $\Id A$ be the set of all ideals of $A$. We  assume that $\Id A$ is equipped with the topology $\tau_w$ for which a base is given by the family of sets of the form
   \[\{I\in\Id A: I\not\supset J \mbox{ for all }J\in F\},\]
    where $F$ is a finite set (possibly empty) of ideals of $A$ (see \cite[Section 2]{Aprimals} and the references cited therein). This topology induces Fell's inner hull-kernel topology on the representation space $\Rep A$ of \cite[p.~206]{KT}.

\subsection{Choice of measures}
Throughout we fix $\mu=\mu_G$ to be the unique Haar measure on the compact group $G$ such that $\mu(G)=1$. Let $\Sigma(G)$ be the space of closed subgroup go $G$ equipped with the Fell topology from \cite{Fell-topology}. Thus $\Sigma(G)$ is a compact Polish space (see \cite[Appendix~H]{tfb^2}). In particular, $\Sigma(G)$ is second countable and sequentially compact.
For every $H\in\Sigma(G)$ we choose $\mu_H$ to be the Haar measure on $H$ such that $\mu_H(H)=1$. Then $H\mapsto \int_H f(s)\, d\mu_H(s)$ is continuous for every $f\in C(G)$, by \cite[Lemma~H8]{tfb^2}.
 This is a very special ``continuous choice of Haar measures on $\Sigma$'' since the constant function $1$ on $G$ has compact support. There exists a
unique quasi-invariant measure $\nu_H$ on $G/H$
such that for $f\in C(G)$,
\begin{equation}\label{eq-double-integral}\int_G f(s)\ d\mu(s)=\int_{G/H}\int_H f(st)
\ d\mu_H(t) \ d\nu_H(sH)
\end{equation}
 (see, for example, \cite[Lemma C.2]{tfb}; we can take the $\rho$ functions appearing there to be identically 1 because $G$ is compact).  In particular, $\nu_H(G/H)=1$.

 By our choice of Haar measures on the closed subgroups of $G$,  integration over $G/H$ may be viewed as integration over $G$: if $f$ is a  function on $G/H$ regarded as a function on $G$  which is constant on $H$-orbits, then
 \[
\int_{G/H} f(sH)\ d\nu_H(sH)
=\int_{G/H}\int_H f(sh)\ d\mu_H(t) \ d\nu_H(sH)
=\int_G f(s)\ d\mu(s).
\]

\subsection{Induced representations}\label{subsec-inducedreps}

Let $H$ be a closed subgroup of the compact group $G$, and let   $Z_H=\overline{C(G, C_c(X))}$ be Green's right-Hilbert  \[(C_0(X)\rtimes_\lt G)-(C_0(X)\rtimes_\lt H)\] bimodule obtained from his $((C_0(X)\otimes C_0(G/H))\rtimes_{\lt\otimes\lt} G)$--$(C_0(X)\rtimes_\lt H)$ imprimitivity bimodule (see \cite[Proposition~3]{green} or \cite[Theorem~4.22]{tfb^2}).
Let  $(\pi, V)$ be a covariant representation of $(C_0(X), H, \lt)$. We write $\Ind_H^G(\pi\rtimes V)$ for the representation of $C_0(X)\rtimes_\lt G$ induced from the representation $\pi\rtimes V$ of $C_0(X)\rtimes_\lt H$ via $Z_H$.

We will usually apply Green's induction process in the following situation.
Suppose that $H$ is contained in the stability subgroup $S_x$ for some $x\in X$.
Let $V$ be a representation of $H$ on a Hilbert space $\H_V$. Define $\pi_{x,V}:C_0(X)\to B(\H_V)$ by \[\pi_{x, V}(f)h=f(x)h\quad\text{for $f\in C_0(X)$ and $h\in \H_V$.}\] Then $(\pi_{x, V}, V)$ is a covariant representation of $(C_0(X), H, \lt)$. We will often write $\pi_{x}$ for $\pi_{x,V}$ when we think no confusion will arise.

Next we recall Mackey's induction process. Let  $V\in\hat H$ and let \label{page-inducedHspace}
\[
L_V:=\lbrace\xi\in C(G,\H_V)\colon \xi(sh)=V_{h^{-1}}(\xi(s))
{\text{\ for\ }}h\in H\text{\ and\ }s\in G\rbrace.\]
Then
\[(\xi \mid \xi'):=\int_{G/H}( \xi(s)\mid\xi'(s))_{\H_V}\ d\nu_H(sH)\]
defines an inner product on $L_V$.
We denote by $L_V^2$ the completion of $L_V$ with respect to the norm induced by the inner product.
The representation $\Ind_H^G V$ of $G$ induced from $V$ is the representation on $L^2_V$ defined by $((\Ind_H^G V)_s\xi)(r)=\xi(s^{-1}r)$ for $\xi\in L_V$.

It is proved, for example in \cite[Proposition~5.4]{tfb^2}, that if $(\pi, V)$ is covariant for $(C_0(X), \lt, H)$, then
$\Ind_H^G(\pi\rtimes V)$  is unitarily equivalent to $\tilde\pi\rtimes \Ind_H^G V$ where \[(\tilde\pi(f)\xi)(r)=\pi(\lt_{r^{-1}}(f))(\xi(r))\quad\text{for $f\in C_0(X)$, $\xi\in L_V$ and $r\in G$.}\]
(In particular, if $X=\{*\}$,  then  $C(*)=\C$ and $C(*)\rtimes_\lt G=C^*(G)$ and  $\Ind_H^G(\pi_*\rtimes V)$  is unitarily equivalent to $\Ind_H^G V$.)


 Even though our set-up  uses the compactness of $G$, the following two lemmas also hold for non-compact groups. In particular, both lemmas apply to the locally proper transformation groups discussed in \S\ref{sec-Cartan}.  Lemma~\ref{lem-induction} says that when  the orbits $G\cdot x=\{g\cdot x:g\in G\}$ are all closed in $X$ (which is the case when $G$ is compact or $(G,X)$ is locally proper), then every irreducible representation of $C_0(X)\rtimes_\lt G$ is unitarily equivalent to one that is induced from an irreducible representation of a stability subgroup $S_x$ of $G$. Its proof can be pieced together from, for example, the  results in \cite{W2}.

\begin{lemma}\label{lem-induction}
Let $(G, X)$ be a transformation group such that all orbits are closed.
\begin{enumerate}
\item\label{lem-induction-item1} Let $\pi\rtimes U:C_0(X)\rtimes_\lt G\to B(\H)$ be an irreducible representation.  Then there exist $x\in X$ and $V\in\hat S_x$ such that $\pi\rtimes U\simeq\Ind_{S_x}^G(\pi_{x}\rtimes V)$.

\item\label{lem-induction-item2} Let $x\in X$ and $V,W\in \hat S_x$. Then $\Ind_{S_x}^G(\pi_{x}\rtimes V)\simeq \Ind_{S_x}^G(\pi_{x}\rtimes W)$ if and only if $V\simeq W$.
\end{enumerate}
\end{lemma}

\begin{lemma}\label{lemma-D} Let $(G,X)$ be a transformation group.
Let $x\in X$, $g\in G$ and $V\in \hat S_x$. Write $g\cdot V$ for the representation of $S_{g\cdot x}$ defined by $(g\cdot V)(t)=V(g^{-1}tg)$.
Then
\[\Ind_{S_x}^G(\pi_{x}\rtimes V)\simeq \Ind_{S_{g\cdot x}}^G(\pi_{g\cdot x}\rtimes (g\cdot V)).\]
\end{lemma}

\begin{proof}
See \cite[Lemma~5.8]{tfb^2} for a proof in greater generality.
\end{proof}

\subsection{Fell's topology for subgroup representations}
Following \cite[\S2]{FellII}, let
\[
Y=\{(H, s):H\in\Sigma(G), s\in H\}.
\]
Then $Y$ is a closed subset of the compact, Hausdorff space $\Sigma(G)\times G$. Thus $Y$ is
compact and Hausdorff. Since $G$ and $\Sigma(G)$ are second-countable, so are $\Sigma(G)\times G$ and its subspace $Y$.  Let $f,g\in C_c(Y)$. Equipped with the operations defined by
\begin{align*}
f*g(H,s)&=\int_H f(H, s) g(H, t^{-1}s)\, d\mu_H(t)\\
f^*(H,s)&=\overline{f(H, s^{-1})},
\end{align*}
$C_c(Y)$ becomes a $*$-algebra. The formula
$\|f\|_{C_c(Y)}=\sup_{H\in \Sigma(G)}\int_H |f(H, s)|\, d\mu_H(s)$ defines a norm on $C_c(Y)$.  A unitary representation $V$ of a closed subgroup $H$ in $G$ on a Hilbert space $\H_V$ can be lifted to a $*$-representation $W^{H,V}$ of $C_c(Y)$ via $W^{H,V}(f)=V(f(H, \cdot))$; notice that $W^{H,V}$ is norm decreasing.  The universal norm on $C_c(Y)$ is
\[\|f\|_u=\sup\{\|\pi(f)\|:\pi\text{\ is a nondegenerate norm-decreasing representation of $C_c(Y)$}\},\]
and Fell's \emph{subgroup $C^*$-algebra} $C^*_S(G)$ is the completion of $C_c(Y)$ in this norm.  Notice that $C^*_S(G)$ is separable because $Y$ is second countable.

Define $S(G)^\wedge=\{(H,V):H\in\Sigma(H), V\in\hat H\}$. The assignment $(H,V) \to W^{H,V}$ induces a bijection from $S(G)^\wedge$ onto $C^*_S(G)^\wedge$ \cite[Lemma~2.8]{FellII}. The pull-back to $S(G)^\wedge$ of the topology on $C^*_S(G)^\wedge$ is called \emph{Fell's topology for subgroup representations} on $S(G)^\wedge$.  Since $C^*_S(G)$ is separable, $C^*_S(G)^\wedge$ is second countable and hence so is Fell's topology on $S(G)^\wedge$.

\subsection{A continuous sum of Hilbert spaces}\label{subsec-cts-field}
Let $H$ be a closed subgroup of the compact group $G$ and $V\in \hat H$.  We now summarise the construction of the  continuous field of Hilbert spaces over $G/H$  and the associated Hilbert space given in  \cite[\S3]{aH}.
The construction simplifies since $G$ is compact.

Define a relation on $G\times \H_V$ by
$(s,v)\simeq (u,w)$   if   and  only   if there exists
$t\in H$ such that $u=st$ and $w=V_{t^{-1}}v$. Then $\simeq$ is an equivalence relation.
Define \[q:G\times\H_V/\!\!\simeq\,\,\to G/H\] by  $q[(s,v)]=sH$.
Each fibre $\H_{sH}:=q^{-1}(sH)$  is a Hilbert space, and
$q_s\colon\H_{sH}\to\H_V$ given by $q_s([s,v])=v$ is a unitary operator.

Let $L_V^2$ be the Hilbert space described in \S\ref{subsec-inducedreps}.
 For  $\xi\in L^2_V$,  define the vector field $\sigma(\xi)$ by
$\sigma(\xi)(sH)=[s,\xi(s)]\in\H_{sH}$.
To form a continuous field $(\lbrace\H_{sH}\rbrace_{sH\in G/H},\Lambda)$ of Hilbert spaces we take the fundamental family  to be
$\Lambda:=\lbrace\sigma(\xi)\colon\xi\in L^2_V\text{\ is continuous} \rbrace$.
(Then $\Lambda$ is complete in the  sense required in \cite{duflo}.)

For $g:G/H\to [0,\infty)$ define  $\int_{G/H}^*g\, d\nu_H=\inf\left\{{\int_{G/H} f\, d\nu_H}\right\}$ where the infimum is taken over measurable functions $f$ dominating $g$ almost everywhere, and  then extend the definition of  $\int^*$  to complex-valued functions by linearity.
For any vector field $\phi$ over $G/H$ set
\begin{align*}
N_2(\phi)^2&=\int^*_{G/H} \|\phi(sH)\|^2_{\H_{sH}} \ d\nu_H(sH).
\end{align*}
Then $\phi$  is \emph{square-summable} if
$N_2(\phi)<\infty$ and $\phi$ is the $N_2$-limit of vector fields in $\Lambda$.
Let $\mathcal{L}^2(\Lambda)$ be the space of all the square-summable vector fields.
A vector field $\phi$ is  \emph{measurable} if it satisfies
Lusin's property: for all compact subsets $K$ of $G/H$ and
$\epsilon>0$, there exists a compact subset $K'$ such that
$\phi_{|K'}\in\Lambda$   and $\mu(K\setminus K')<\epsilon$.
Thus   a vector field $\phi$ is an element of $\mathcal{L}^2(\Lambda)$ if and
only if $N_2(\phi)<\infty$ and $\phi$ is measurable.

Let $N=\lbrace \phi\in\Pi_{sH\in G/H}\H_{sH}\colon \phi=0\ \nu_H\text{-a.e.}\rbrace$.
Then  $L^2(\Lambda):=\mathcal{L}^2(\Lambda)/N$ is a Hilbert space, called the continuous sum over $G/H$,
with inner product
\begin{equation*}( \phi_1\mid \phi_2)_{L^2(\Lambda)}
=\int_{G/H}^*(
\phi_1(sH)\mid \phi_2(sH))_{\H_{sH}}\ d\nu_H(sH)\quad\quad
(\phi_1,\phi_2\in\mathcal{L}^2(\Lambda)).
\end{equation*}
We will write $\phi$  for the class of $\phi$ in $L^2(\Lambda)$.
It is easily seen that the map $\sigma\colon L^2_V\to L^2(\Lambda)$
is a unitary operator.

\section{Trace formulae}\label{sec-trace}

In this section we establish a number of trace formulae for later use.

\begin{prop}\label{prop-long-calculation}
Let $(G,X)$ be a second-countable transformation group with $G$ compact.
Let $x\in X$, $H$ a closed subgroup of $G$ with $H\subset S_x$ and  $V\in\hat H$. Let $a\in C(G, C_c(X))^+\subset C_0(X)\rtimes_\lt G$. Then the operator $\Ind_H^G(\pi_{x}\rtimes V)(a) $ is trace class and
\begin{equation}\label{eq-calculate-trace2}
\tr\big( \Ind_H^G(\pi_{x}\rtimes V)(a) \big)=
\int_{G/H} \int_H   a(rt^{-1}r^{-1})( r\cdot x)\tr(V_t^*)\, d\mu_H(t)\, d\nu_H(rH).
\end{equation}
\end{prop}
\begin{proof}
We will use a vector-valued version of Mercer's Theorem to compute the trace in \eqref{eq-calculate-trace2}. This theorem is obtained by combining  Remarque~3.2.1, Th\'eor\`eme~3.3.1 and  Proposition~3.1.1 of \cite{duflo}; there is quite a lot of work to be done to  check that these results apply.

Let $(\lbrace\H_{sH}\rbrace_{sH\in G/H},\Lambda)$ be the continuous field of Hilbert spaces over $G/H$  and $L^2(\Lambda)$ the continuous sum described in \S\ref{subsec-cts-field}.
We start by realising $\Ind_H^G(\pi_{x}\rtimes V)(a)$ as a kernel operator on $L^2(\Lambda)$.
First, $\Ind_H^G(\pi_{x}\rtimes V)$ is unitarily equivalent to $\tilde\pi_{x}\rtimes\Ind_H^G V$ on $L_V^2$ where $(\tilde\pi_{x}(f)\xi)(r)=f(r\cdot x)\xi(r)$ and $((\Ind_H^G V)_s\xi)(r)=\xi(s^{-1}r)$.  We compute, for $\xi\in L_V\subset L_V^2$ and $r\in G$:
\begin{align*}
\big((\tilde\pi_{x}\rtimes \Ind_H^G V)(a)\xi \big)(r)&=\Big(\int_G \tilde\pi_{x}(a(s))(\Ind_H^G V)_s\xi\, d\mu(s)\Big)(r)\\
&=\int_G a(s)( r\cdot x)\xi(s^{-1}r)\, d\mu(s)\\
&=\int_G a(ru^{-1})(r\cdot x)\xi(u)\, d\mu(u)\\
\intertext{which, using \ref{eq-double-integral}, becomes}
&=\int_{G/H}\int_H a(rt^{-1}u^{-1})( r\cdot x)\xi(ut)\, d\mu_H(t)\, d\nu_H(uH)\\
&=\int_{G/H}\Big(\int_H a(rt^{-1}u^{-1})( r\cdot x)V_t^{-1}\, d\mu_H(t)\Big)\xi(u) \, d\nu_H(uH),\\
&=  \int_{G/H}K_x(r, u)\xi(u) \, d\nu_H(uH)
\end{align*}
where \[K_x(r, u)=\int_H a(rt^{-1}u^{-1})( r\cdot x)V_t^{-1}\, d\mu_H(t)\in B(\H_V).\]  We note here, for future use, that the calculation above implicitly shows that for fixed $r,u\in G$, the function $t\mapsto  a(rt^{-1}u^{-1})( r\cdot x)V_t^{-1}$ is a $\mu_H$-measurable $B(L_V^2)$-valued function.

The maps $\sigma\colon L^2_V\to L^2(\Lambda)$ and $q_r\colon\H_{rH}\to\H_V$ defined by $\sigma(\xi)(rH)=[r,\xi(r)]\in\H_{rH}$ and $q_r([r,v])=v$, respectively, are unitary operators. Let \begin{equation*}\tilde K_x(rH, uH)=q_r^{-1}K_x(r,u)q_u.\end{equation*}
We claim that $\sigma (\tilde\pi_{x}\rtimes \Ind_H^G V)(a)\sigma^*$ is the operator in $B(L^2(\Lambda))$ determined by the kernel
$\tilde K_x$ in the sense of \cite[\S3.1]{duflo}. Thus we need to verify that
\begin{enumerate}
\item\label{kernel-enumerate1} $\tilde K_x$ is a well-defined function on $G/H\times G/H$;
\item\label{kernel-enumerate2}
for $\phi_1, \phi_2\in\Lambda$,
the function \begin{equation}\label{eq-kernel}(rH, uH)\mapsto ( \tilde K_x(rH,uH)(\phi_1(uH))\mid \phi_2(rH))_{\H_{rH}}\end{equation} is $\nu_H\times\nu_H$-integrable and
\begin{align*}&( (\sigma (\tilde\pi_{x}\rtimes U_V)(a)\sigma^*)\phi_1\mid \phi_2)_{L^2(\Lambda)}\\
&=\int_{G/H}\int_{G/H}\big(
\tilde K_x(rH,uH)(\phi_1(uH))|\phi_2(rH)\big)_{\H_{rH}} \ d\nu_H(uH)\ d\nu_H(rH).\end{align*}
\end{enumerate}

To see that $\tilde K_x$ is well-defined, we observe that
$q_{uh}=V_h^{-1}q_u$ and $K_x(rh,uk)=V_h^{-1}K_x(r,u)V_k$ for $h,k\in H$ and $r,u\in G$.
Thus
\[
\tilde K_x(rhH, ukH)=q_{rh}^{-1} K_x(rh, uk)q_{uk}=q_r^{-1}V_hV_h^{-1}K_x(r,u)V_kV_k^{-1}q_u=\tilde K_x(rH, uH),
\]
giving~\eqref{kernel-enumerate1}.
For~\eqref{kernel-enumerate2}, we begin by showing that $\tilde K_x$ is continuous in the sense that the function at \eqref{eq-kernel} is continuous.
Write $\phi_i=\sigma(\xi_i)$ where $\xi_i\in L^2_V$ is continuous for $i=1,2$.  Then $q_u(\sigma(\xi_i)(uH))=q_u([(u, \xi_i(u))])=\xi_i(u)$ for $u\in G$.  Thus
\begin{align*}
( \tilde K_x(rH,uH)(\phi_1(uH))\mid \phi_2(rH))_{\H_{rH}}
&=( K_x(r,u)q_u(\phi_1(uH))\mid q_r(\phi_2(rH)))_{\H_V}\\
&=(K_x(r,u)(\xi_1(u))\mid\xi_2(r))_{\H_V}.
\end{align*}
Since the canonical map $G\to G/H$ is open and the $\xi_i$ and the inner product are continuous, it suffices to show that $K_x$ is a continuous function on $G\times G$. So suppose that $r_n\to r$ and $u_n\to u$ in $G$.
Set $F(t)=a(rt^{-1}u^{-1})( r\cdot x)V_t^{-1}$ and $F_n(t)=a(r_nt^{-1}u_n^{-1})( r_n\cdot x)V_t^{-1}$. As observed above, $F$ and $F_n$ are $\mu_H$-measurable. Since $a$, regarded as a function on $G\times X$,  is jointly continuous, $F_n\to F$ pointwise. Since $H$ is compact, $t\mapsto \|F_n(t)\|$ is dominated by the function $t\mapsto \|a\|_\infty$ in $L^1(H)$, where $\|a\|_\infty=\max_{s\in G}\|a(s)\|_\infty$.  The vector-valued dominated convergence theorem (see, for example, \cite[Proposition~B.32]{tfb^2}) now implies that $K_x(r_n,u_n)=\int_H F_n(t)\, d\mu_H(t)\to \int_H F(t)\, d\mu_H(t)=K_x(r,u)$.  Thus $K_x$ is continuous. Hence the function at \eqref{eq-kernel} is continuous and therefore $\nu_H\times\nu_H$-integrable since $G/H$ is compact. The displayed equation in~\eqref{kernel-enumerate2} follows from the definition of the inner product in $L^2_V$ together with the earlier calculation that
\[\big((\tilde\pi_{x}\rtimes \Ind_H^G V)(a)\xi_1 \big)(r)=\int_{G/H}K_x(r, u)\xi_1(u) \, d\nu_H(uH).\]

Now we need to argue that the results of \cite{duflo} apply. Since $\tilde K_x$ is a continuous kernel defining a positive operator, $\tilde K_x(rH, rH)$ is positive for all $rH\in G/H$ by the first part of \cite[Th\'eor\`eme~3.3.1]{duflo}.
Next we compute $\int_{G/H}^*\tr(\tilde K_x(rH, rH))\, d\nu_H(rH)$.  For this, let $\{\eta_i\}_{i=1}^k$ be an orthonormal basis for $\H_V$. Then
\[rH\mapsto \tr(\tilde K_x(rH,rH))=\tr(K_x(r,r))=\sum_{i=1}^k(K_x(r,r)\eta_i\mid\eta_i)_{\H_V}
\]
is a continuous function on $G/H$ and hence is $\nu_H$-measurable. Thus
\begin{align}
\int_{G/H}^*\tr(\tilde K_x&(rH, rH))\, d\nu_H(rH)=\int_{G/H}\tr(\tilde K_x(rH, rH))\, d\nu_H(rH)\notag\\
&=\int_{G/H}\sum_{i=1}^k(K_x(r,r)\eta_i\mid\eta_i)_{\H_V}\, d\nu_H(rH)\notag\\
&=\int_{G/H} \int_H \sum_{i=1}^k a(rt^{-1}r^{-1})( r\cdot x)(V_t^{-1}\eta_i\mid \eta_i)_{\H_V}\, d\mu_H(t)\, d\nu_H(rH)\label{eq-calculate-trace} \\
&\leq k\|a\|_\infty\mu_H(H)\nu_H(G/H)<\infty.\notag
\end{align}

We write $\H_{rH,uH}$ for the operators from $\H_{uH}$ to $\H_{rH}$, which are all Hilbert-Schmidt because each $\H_{rH}$ is finite dimensional. Then $(\lbrace\H_{rH,uH},\Sigma:=\Lambda\otimes\Lambda)$ is a continuous field of Hilbert spaces.
(Here $\Sigma$ is the fundamental family generated by $(rH, uH)\mapsto \phi(rH)\otimes \overline{\psi(uH)}$, where $\phi,\psi\in\Lambda$ and  $\phi(rH)\otimes \overline{\psi(uH)}$ is the rank-one operator $v\mapsto (v\mid \psi(uH))_{\H_{uH}}\phi(rH)$ for $v\in \H_{uH}$.)  Since $G$ is second-countable and $\H_V$ is  finite-dimensional, hence separable, the fundamental family $\Lambda$ is countable in the sense that there is a sequence $\{\phi_n\}\subset \Lambda$ such that $\{\phi_n(rH)\}$ is dense in $\H_{rH}$ for all $rH\in G/H$.  Remarque~3.2.1 in \cite{duflo} says, first, that $\Sigma$ is countable as well, and second, that $\tilde K_x$ is measurable as a vector field on $G/H\times G/H$.  Now the third part of \cite[Th\'eor\`eme~3.3.1]{duflo} applies and says that the operator $\sigma (\tilde\pi_{x}\rtimes \Ind_H^G V))(a)\sigma^*$ defined by $\tilde K_x$ is trace class. Finally,  by \cite[Proposition~3.1.1]{duflo}, $\tilde K_x(rH, rH)$ is trace class for almost all $rH\in G/H$, and
\[
\tr(\sigma (\tilde\pi_{x}\rtimes \Ind_H^G V))(a)\sigma^*)=\int_{G/H}\tr(\tilde K_x(rH, rH))\, d\nu_H(rH),
\]
which we calculated at \eqref{eq-calculate-trace} above.  Finally, we note that
\[
\eqref{eq-calculate-trace} =
\int_{G/H} \int_H   a(rt^{-1}r^{-1})( r\cdot x)\tr(V_t^*)\, d\mu_H(t)\, d\nu_H(rH)\qedhere
\]
\end{proof}

We thank Iain Raeburn for pointing us to the following lemma.

\begin{lemma}\label{lem-iain}  Let $(G,X)$ be a second-countable transformation group.  Let $z\in X$, $H$ be a closed subgroup of $G$ with $H\subset S_z$ and $k\in\NN\cup\{\infty\}$.   Suppose $W=\oplus_{i=1}^k W_i:H\to U(\H_W)$ is a unitary representation of $H$.  Then
\[
\Ind_H^G(\pi_{z,W}\rtimes W)\simeq\bigoplus_i \Ind_H^G(\pi_{z,W_i}\rtimes W_i).
\]
\end{lemma}
\begin{proof}
Here $\pi_{z,W}\rtimes W$ is unitarily equivalent to $\oplus_i \pi_{z,W_i}\rtimes W_i$. The lemma follows because induction preserves direct-sum decompositions.
\end{proof}

For two representations $V$ and $W$ of a subgroup $H$ of $G$, we  write $[V:W]$ for the multiplicity of $W$ in $V$.

\begin{prop}\label{prop-sum-traces}  Let $(G,X)$ be a second-countable transformation group with $G$ compact.   Let $z\in X$,  $H$ a closed subgroup of $G$ such that $H\subset S_z$, and $V\in \hat H$. Then
\[
\Ind_H^G(\pi_{z,V}\rtimes V)\simeq\bigoplus_{W\in \hat S_z}\big[W|_H:V\big]\, \Ind_{S_z}^G(\pi_{z,W}\rtimes W).
\]
\end{prop}

\begin{proof}
Since $H\subset S_z\subset G$, by induction in stages for the induction process $\Ind_H^G$ (\cite[Theorem~5.9]{tfb^2}), we have
\begin{align*}\label{eq-trace2}
\Ind_H^G(\pi_{z,V}\rtimes V)&
\simeq \Ind_{S_z}^G\Big(\Ind_H^{S_z}(\pi_{z,V}\rtimes V)\Big).
\end{align*}
The key now is to observe that we can realise $\Ind_H^{S_z}(\pi_{z,V}\rtimes V)$ on the Hilbert space $L_V^2$ of $\Ind_H^{S_z} V$ as $\tilde\pi_{z,V}\rtimes\Ind_H^{S_z} V$ and that $\tilde\pi_{z,V}=\pi_{z,V}$. To see the latter, we compute for  $\xi\in L_V$ and $r\in S_z$:
 \[(\tilde\pi_{z,V}(f)\xi)(r)=f(r\cdot z)\xi(r)=f(z)\xi(r)=(\pi_{z,V}(f)\xi)(r). \]
Thus $\tilde\pi_{z,V}=\pi_{z,V}$. Now
\begin{align}\label{eq-trace3}
\Ind_H^G(\pi_{z,V}\rtimes V)&\simeq\Ind_{S_z}^G\Big( \tilde\pi_{z,V}\rtimes\Ind_H^{S_z} V \Big)\notag\\
&= \Ind_{S_z}^G\Big( \pi_{z,V}\rtimes\Ind_H^{S_z} V \Big).
\end{align}
The induced representation $\Ind_H^{S_z} V$ of the compact group $S_z$ is a direct sum of irreducible representations. Using Frobenius Reciprocity (see, for example, \cite[Theorem~7.4.1]{DE}) we get
\[
\Ind_H^{S_z} V=\bigoplus_{W\in \hat S_z}\big[\Ind_H^{S_z} V:W\big]\, W=\bigoplus_{W\in \hat S_z}\big[W|_H:V\big]\, W.
\]
Since $S_z$ is second-countable and compact,  $\hat S_z$ is second-countable and discrete. Thus the direct sum above is countable.
Now
\begin{align*}
\eqref{eq-trace3}&=\Ind_{S_z}^G\Big(\pi_{z,V}\rtimes\Big(\bigoplus_{W\in \hat S_z}\big[W|_H:V\big]\, W \Big)\Big)\\
&\simeq\bigoplus_{W\in \hat S_z}\big[W|_H:V\big]\, \Ind_{S_z}^G(\pi_{z,W}\rtimes W)
\end{align*}
using Lemma~\ref{lem-iain}.
\end{proof}

The following result replaces an earlier version (involving a constant sequence of stability subgroups and representations) which sufficed for the proof of Theorem~\ref{thm-pss}. We are grateful to the referee for suggesting the use of Fell's topology for subgroup representations in this more general version.

\begin{prop}\label{prop-traces-converging-new-new} Let $(G,X)$ be a second-countable transformation group with $G$ compact. Suppose that $x_n\to z$ in $X$ and that $(S_{x_n},V_n)\to (H, V)$ in Fell's topology for subgroup representations, where $H$ is a closed subgroup of $G$, $V\in \hat H$ and $V_n\in\widehat{S_{x_n}}$ ($n\geq1$). Then $H\subset S_z$ and, for all $a\in C(G, C_c(X))^+\subset C_0(X)\rtimes_\lt G$,
\begin{align}
\tr\big( \Ind_{S_{x_{n}}}^G(\pi_{x_{n}}\rtimes V_n)(a) \big)
&\to
\tr\big( \Ind_H^G(\pi_{z}\rtimes V)(a) \big)\label{eq-traces-converging-1}\\
&=\sum _{W\in \hat S_z}[W|_H:V]\,\tr \big(\Ind_{S_z}^G(\pi_{z,W}\rtimes W) (a)\big).\label{eq-traces-converging-2}
\end{align}
\end{prop}

\begin{proof}
Let $a\in C(G, C_c(X))^+$. By a standard argument of general topology, in order to establish \eqref{eq-traces-converging-1}, it suffices to consider an arbitrary subnet $(x_{n_{\beta}})_{\beta}$ of $(x_n)$ and to show that there is a subnet of this subnet for which the corresponding convergence of traces holds.
By \cite[Lemma on page~205]{Bag}, we may pass to a subnet of $(x_{n_{\beta}})_{\beta}$ and relabel so that $S_{x_{n_{\beta}}}\to H $ in $\Sigma(G)$ and $\tr(V_{n_{\beta}}(t_{n_{\beta}}))\to \tr(V(t))$ whenever $t_{n_{\beta}}\to_{\beta} t$ in $G$ with $t\in H$ and $t_{n_{\beta}}\in S_{x_{n_{\beta}}}$ for all $\beta$. In particular, since $S_{x_{n_{\beta}}}\to H $ and $G$ acts continuously on the Hausdorff space $X$, it follows that $H\subset S_z$ and hence that  $(\pi_z,V)$ is indeed a covariant pair.

For each $n\geq1$, $t\in S_{x_n}$ and $r\in G$, let \[E_n(t,r)= a(rt^{-1}r^{-1})( r\cdot x_n)\tr(V_n(t)^*).\]
Then $E_n(t,r)$ is precisely the integrand at \eqref{eq-calculate-trace2} with $x$ replaced by $x_n$, $H$ replaced by $S_{x_n}$, and $V$ replaced by $V_n$. Since $a$ is continuous and the finite-dimensional representation $V_n$ is norm-continuous, $E_n$ is continuous on $S_{x_n}\times G$.
Similarly, for $t\in H$ and $r\in G$,  set
\[
E(t, r)=a(rt^{-1}r^{-1})( r\cdot z)\tr(V(t)^*).
\]
Then $E$ is continuous on $H\times G$. Suppose that $(t_{n_{\beta}}, r_{\beta}) \to (t,r)$ in $G\times G$ where $(t,r)\in H\times G$ and
$(t_{n_{\beta}}, r_{\beta}) \in S_{x_{n_{\beta}}}\times G$ for each $\beta$. Then $E_{n_{\beta}}(t_{n_{\beta}}, r_{\beta}) \to E(t,r)$ by the continuity of $a$ and the convergence of $\tr(V_{n_{\beta}}(t_{n_{\beta}}))$ to $\tr(V(t))$. Furthermore, since $S_{x_{n_{\beta}}}\to H $ in $\Sigma(G)$, we have $S_{x_{n_{\beta}}}\times G \to H\times G $ in $\Sigma(G\times G)$.
 By \cite[Proposition~4 on page~184]{Bag} and Fubini's theorem
\[\int_G\int_{S_{x_{n_{\beta}}}} E_{n_{\beta}}(t,r)\, d\mu_{S_{x_{n_{\beta}}}}(t)\, d\mu(r)\to \int_G\int_H E(t,r)\, d\mu_H(t)\, d\mu(r).\]
By our choice of Haar measures on the subgroups,
\[\int_{G/S_{x_{n_{\beta}}}}\int_{S_{x_{n_{\beta}}}} E_{n_{\beta}}(t,r)\, d\mu_{S_{x_{n_{\beta}}}}(t)\, d\nu_{S_{x_{n_{\beta}}}}(rS_{x_{n_{\beta}}})\to \int_{G/H}\int_H E(t,r)\, d\mu_H(t)\, d
\nu_H(rH).\]
Hence Proposition~\ref{prop-long-calculation} gives \eqref{eq-traces-converging-1} with $n$ replaced by $n_{\beta}$. As indicated at the outset, this suffices to establish \eqref{eq-traces-converging-1}.
Then \eqref{eq-traces-converging-2} follows by Proposition~\ref{prop-sum-traces}.
\end{proof}

\begin{prop}\label{ref-propn2} Let $(G,X)$ be a second-countable transformation group with $G$ compact. Suppose that $x_n\to z$ in $X$ and suppose that  $(S_{x_n},R_n)\to (H, R)$ in Fell's topology for subgroup representations, where $H$ is a closed subgroup of $G$, $R\in \hat H$ and $R_n\in\widehat{S_{x_n}}$ ($n\geq1$).
  Then $H\subset S_z$ and, for all $V\in \hat S_z$,
\[M_U\big(\Ind_{S_z}^G(\pi_{z}\rtimes V)\big)\geq \big[ V|_H:R \big].\]
\end{prop}

\begin{proof}
Write $\rho_n=\Ind_{S_{x_n}}^G(\pi_{x_n}\rtimes R_n)$. Let $a\in C(G, C_c(X))^+\subset C_0(X)\rtimes_\lt G$. By Proposition~\ref{prop-traces-converging-new-new}, $H\subset S_z$ and
\begin{equation*}
\tr\big( \rho_n(a) \big)
\to
\sum _{W\in \hat S_z,\ [W|_H:R]\neq 0}[W|_H:R]\,\tr \big(\Ind_{S_z}^G(\pi_{z,W}\rtimes W) (a)\big).
\end{equation*}
Let $V\in \hat S_z$. If $[V|_H:R]=0$, then the result follows trivially.
Otherwise,  $[V|_H:R]\geq 1$. Then Theorem~4.1 of \cite{ASS} implies that $\Ind_{S_z}^G(\pi_{z}\rtimes V)$ is a limit of $(\rho_n)$ and
\[
\big[V|_H:R\big]= M\big( \Ind_{S_z}^G(\pi_{z}\rtimes V), (\rho_n) \big)\leq M_U\big(\Ind_{S_z}^G(\pi_{z}\rtimes V)\big).\qedhere\]
\end{proof}

\begin{lemma}\label{lemma-claim} Let $(G,X)$ be a proper second-countable transformation group.  Suppose that $\Ind_{S_{x_n}}^G(\pi_{x_n}\rtimes R_n)\to \Ind_{S_z}^G(\pi_{z}\rtimes V)$ in $(C_0(X)\rtimes G)^\wedge$.
By passing to a subsequence and moving within orbits, we may assume that $x_n\to z$ in $X$, and that there exist a closed subgroup $H\subset S_z$ and an irreducible subrepresentation $R$ of $V|_H$ such that $(S_{x_n}, R_n)\to (H, R)$ in Fell's topology for subgroup representations.
\end{lemma}

\begin{proof} We start by describing the homeomorphism of \cite[Theorem~5.3.3]{Neu} of a quotient space of\[
\Stab(X)^\wedge:=\{(x, S_x, W): x\in X, W\in \hat{S_x}\}
\]
onto the spectrum of $C_0(X)\rtimes G$. The topology on $\Stab(X)^\wedge$ is defined in terms of a closure operation in \cite[Definition~2.3.1 and Proposition~2.3.3]{Neu}. For $A\subset \Stab(X)^\wedge$, the closure of $A$ consists of all $(z, S_z, V)\in \Stab(X)^\wedge$ for which there exist a closed subgroup $H\subset S_z$, an irreducible subrepresentation $R$ of $V|_H$ and a net $(x_\alpha, S_{x_\alpha}, R_\alpha)$ in $A$ such that $x_\alpha\to z$ in $X$ and $(S_{x_\alpha}, R_\alpha)\to (H, R)$ in Fell's topology for subgroup representations.  It follows from \cite[Theorem~6.1.4]{Neu} that $\Stab(X)^\wedge$ is homeomorphic to the spectrum of a separable $C^*$-algebra and hence is second countable.

The map $(g, (x, S_x, W))\mapsto (g\cdot x, S_{g\cdot x}, g\cdot W)$ gives a continuous action of $G$ on $\Stab(X)^\wedge$. Notice that by Lemma~\ref{lemma-D}, $\Ind_{S_x}^G(\pi_{x}\rtimes W)\simeq \Ind_{S_{g\cdot x}}^G(\pi_{g\cdot x}\rtimes (g\cdot W))$.  It is well-known that the map $(x, S_x, V)\mapsto \Ind_{S_x}^G(\pi_{x}\rtimes W)$ factors through a bijection $\Ind^G$ of $G\backslash \Stab(X)^\wedge$ onto the spectrum of $C_0(X)\rtimes G$.
The content of \cite[Theorem~5.3.3]{Neu} is that the assignment $[(x, S_x, V)]\mapsto \Ind_{S_x}^G(\pi_{x}\rtimes W)$ defines a homeomorphism from  $G\backslash \Stab(X)^\wedge$ with the quotient topology onto the spectrum of $C_0(X)\rtimes G$ (the main work involves showing that the map is open).

Since $\Ind_{S_{x_n}}^G(\pi_{x_n}\rtimes R_n)\to \Ind_{S_z}^G(\pi_{z}\rtimes V)$,  we have $[(x_n, S_{x_n}, R_n)]\to [(z, S_z, V)]$ in $G\backslash \Stab(X)^\wedge$. The quotient map $\Stab(X)^\wedge\to G\backslash \Stab(X)^\wedge$ is open and $\Stab(X)^\wedge$ is second -countable, and hence there is a subsequence $((x_{n_k}, S_{x_{n_k}}, V_{n_k}))_k$ of $((x_{n}, S_{x_{n}}, V_{n}))_n$ and a sequence $(g_k)$ in $G$ such that $g_k\cdot (x_{n_k}, S_{x_{n_k}}, R_{n_k})=(g_k\cdot x_{n_k}, S_{g_k\cdot x_{n_k}}, g_k\cdot R_{n_k})\to (z, S_z, V)$ in $\Stab(X)^\wedge$. But $\Ind_{S_{x_{n_k}}}^G(\pi_{x_{n_k}}\rtimes R_{n_k})\simeq  \Ind_{S_{g_k\cdot x_{n_k}}}^G(\pi_{g_k\cdot x_{n_k}}\rtimes g_k\cdot R_{n_k})$ by Lemma~\ref{lemma-D}. Thus we may replace $x_{n_k}$ by $g_k\cdot x_{n_k}$ and $R_{n_k}$ by $g_k\cdot R_{n_k}$, and relabel, and then assume that $(x_n, S_{x_n}, R_n)\to (z, S_z, V)$ in $\Stab(X)^\wedge$.
Now by the definition of the closure operation in $\Stab(X)^\wedge$, there exist a closed subgroup $H\subset S_z$, an irreducible subrepresentation $R$ of $V|_H$ and a subnet $(x_{n_\alpha}, S_{x_{n_\alpha}}, R_{n_\alpha})$ such that $x_{n_\alpha}\to z$ in $X$ and $(S_{x_{n_\alpha}}, R_{n_\alpha})\to (H, R)$ in Fell's topology for subgroup representations. But both $X$ and $S(G)^\wedge$ are second countable, and a routine argument gives a subsequence $(x_{n_j})$ (indexed by $j\in\NN$) of  $(x_n)$ such that $x_{n_j}\to z$ and  $(S_{x_{n_j}},R_{x_{n_j}})\to (H,R)$.
\end{proof}

\begin{thm}\label{thm-new-general} Let $(G,X)$ be a second-countable transformation group with $G$ compact. Let $z\in X$ and $V\in \hat S_z$.
 Then there exist a closed subgroup $H$ of $G$ with $H\subset S_z$ and an irreducible subrepresentation $R$ of $V|_H$
such that
\[M_U\big(\Ind_{S_z}^G(\pi_{z}\rtimes V)\big) =[ V|_H:R \big] \leq \dim V < \infty.\]
Moreover, if the stabiliser map $X\to \Sigma(G)$, $x\to S_x$, is continuous at $z$ then
$$M_U\big(\Ind_{S_z}^G(\pi_{z}\rtimes V)\big) =1.$$
\end{thm}

\begin{proof}
Set $\rho:=\Ind_{S_z}^G(\pi_{z}\rtimes V)$. By \cite[Lemma~1.2]{AK} there exists a sequence $(\rho_n)$  such that the upper multiplicity $M_U(\rho)$  equals the common upper and lower multiplicity $M(\rho, (\rho_n))$ of $\rho$ relative to the sequence $(\rho_n)$. Implicit in this assertion is that $\rho_n\to\rho$ because the lower relative multiplicity is nonzero (see \cite[\S2]{AS}). We note here, for multiple future use, that if $(\rho_{n_k})$ is a subsequence of $(\rho_n)$, then  $M(\rho, (\rho_{n_k}))=M(\rho, (\rho_n))$ (see \cite[Equation~2.3]{AS}), and that the multiplicity numbers depend only  on the class of the irreducible representation.  For each $n$ there exists $x_n\in X$ and $R_n\in\hat S_{x_n}$ such that $\rho_n=\Ind_{S_{x_n}}^G(\pi_{x_n}\rtimes R_n)$.

By Lemma~\ref{lemma-claim}, by passing to a subsequence and moving within orbits, we may assume that $x_n\to z$ in $X$, and that there exist a closed subgroup $H\subset S_z$ and an irreducible subrepresentation $R$ of $V|_H$ such that $(S_{x_n}, R_n)\to (H, R)$ in Fell's topology for subgroup representations.

By Proposition~\ref{prop-traces-converging-new-new}, for all $a\in C(G, C_c(X))^+\subset C_0(X)\rtimes_\lt G$ we have
\[
\tr\big( \Ind_{S_{x_n}}^G(\pi_{x_{n}}\rtimes R_n)(a) \big)
\to
\sum _{W\in \hat S_z}[W|_H:R]\,\tr \big(\Ind_{S_z}^G(\pi_{z,W}\rtimes W) (a)\big).
\]
Now Theorem~4.1 of \cite{ASS} implies that \[\big[V|_H:R\big]= M\big( \rho, (\rho_n) \big)= M_U(\rho).\]

Finally, suppose that the stabiliser map is continuous at $z$. Then $S_{x_n} \to S_z$ in $\Sigma(G)$.  On the other hand, it follows from \cite[Lemma 2.5]{FellII} that $S_{x_n} \to H$. Since $\Sigma(G)$ is Hausdorff, $H=S_z$. Then the irreducibility of $V$ implies that
$M_U\big(\Ind_{S_z}^G(\pi_{z}\rtimes V)\big) =1$. \qedhere
\end{proof}

\begin{cor}\label{cor-dim1} Let $(G,X)$ be a second-countable transformation group with $G$ compact.
\begin{enumerate}
\item Let $z\in X$ and $V\in {\hat S_z}$ be a character.  Then \[M_U\big(\Ind_{S_z}^G(\pi_{z}\rtimes V)\big)=1.\]
\item Suppose that $S_z$ is abelian for all $z\in X$. Then $C_0(X)\rtimes_{\rm lt}G$ is a Fell algebra.
\end{enumerate}
\end{cor}
\begin{proof}
(1) By Theorem~\ref{thm-new-general} there exist a closed subgroup $H$ of $S_z$ and an irreducible subrepresentation $R$ of $V|_H$ such that
\[M_U\big(\Ind_{S_z}^G(\pi_{z}\rtimes V)\big)=\big[ V|_H:R \big].\]
The right-hand side is $1$ because $V$ has dimension $1$.

(2) It follows from (1) and Lemma~\ref{lem-induction} that every irreducible representation of $C_0(X)\rtimes_{\rm lt}G$ has upper multiplicity $1$. Hence $C_0(X)\rtimes_{\rm lt}G$ is a Fell algebra by \cite[Theorem 4.6]{A}.
\end{proof}

\section{Main results}
We now turn to the two situations described in the introduction: when $G$ has a principal stability subgroup or when  $X$ is  locally of finite $G$-orbit type.
Our main theorem is:

\begin{thm}\label{thm-pss} Let $(G,X)$ be a second-countable transformation group with $G$ compact. Let $z\in X$ and $V\in \hat S_z$. Suppose that
either
\begin{enumerate}
\item\label{thm-pss-item1}  $G$ contains a principal stability subgroup $S$ or
\item\label{thm-pss-item2}  $X$ is locally of finite $G$-orbit type.
\end{enumerate}
 Then there exists a closed subgroup $H$ of $G$ with $H\subset S_z$
such that
$$M_U\big(\Ind_{S_z}^G(\pi_{z}\rtimes V)\big)=\max_{R\in\hat H}\big[ V|_H:R \big]\leq \dim V<\infty.$$
Furthermore, in case \eqref{thm-pss-item1} the subgroup $H$ may be chosen conjugate to $S$.
\end{thm}

Before we can prove Theorem~\ref{thm-pss}, we need several lemmas.

Let $\pi\rtimes U\in (C_0(X)\rtimes_\lt G)^\wedge$. As shown in the proof of Lemma~\ref{lem-induction}, there exists $x_{(\pi,U)}\in X$ such that $\ker\pi$ consists of the functions vanishing on the orbit $G\cdot x_{(\pi, U)}$.
Since $G$ is compact, the orbits in $X$ are closed. Furthermore, the stability subgroups of $G$ are compact and hence amenable and liminal. Thus $C_0(X)\rtimes_\lt G$ is liminal by \cite[Theorem~3.1]{W2} and so $\Prim(C_0(X)\rtimes_\lt G)$ and $(C_0(X)\rtimes_\lt G)^\wedge$ are homeomorphic.  Let $q:X\to X/G$ be the quotient map. It now follows that the function $\Phi$ of \cite[Theorem~4.8]{GL} from $\Prim(C_0(X)\rtimes_\lt G)$ to the T${}_0$-isation $(X/G)^\sim$ of $X/G$ can be viewed as a function $\Phi:(C_0(X)\rtimes_\lt G)^\wedge \to X/G$, and then $\Phi(\pi\rtimes U)=q(x_{(\pi, U)})$. In particular,
\begin{equation*}\label{eq-Phi} \Phi(\Ind_{S_z}^G(\pi_z\rtimes V))=\Phi(\tilde\pi_z\rtimes\Ind^G_{S_z}V)=q(z),
\end{equation*}
since $\ker(\tilde\pi_z)=\{f\in C_0(X): f(G\cdot z)=\{0\}\}$.
Since $G$ is amenable it follows from \cite[Theorem~4.8]{GL} that  $\Phi$  is open.

\begin{lemma}\label{dense-in-spectrum} Let $(G,X)$ be a second-countable transformation group with $G$ compact. Suppose that $Y$ is a dense subset of $X$.
Then \[\Ind Y:=\{ \Ind_{S_y}^G(\pi_{y}\rtimes U) :y\in Y, U\in \hat S_y \}\] is dense in $(C_0(X)\rtimes_\lt G)^\wedge$.
\end{lemma}

\begin{proof}
Let $O$ be a non-empty open subset of $(C_0(X)\rtimes_\lt G)^\wedge$. Then $q^{-1}(\Phi(O))$ is non-empty and open in $X$. Since $Y$ is dense in $X$, there exists $y\in q^{-1}(\Phi(O))\cap Y$. Now $q(y)=\Phi(\rho)$ for some $\rho\in O$.  By Lemma~\ref{lem-induction}, $\rho=\Ind_{S_z}^G(\pi_{z}\rtimes V)$ for some $z\in X$ and $V\in \hat S_z$.  Then $q(y)=\Phi(\rho)=q(z)$ and so there exists $g\in G$ such that $y=g\cdot z$. By Lemma~\ref{lemma-D}, $\rho=\Ind_{S_y}(\pi_y\rtimes g\cdot V)\in \Ind Y$ and so $O\cap \Ind Y\neq \emptyset$, as required.
\end{proof}

\begin{lemma}\label{lem-park}
Let $(G,X)$ be a second-countable transformation group with $G$ compact.  Let $z, x_n\in X$ for $n\geq 1$, $V\in \hat S_z$, $Q_n\in \hat S_{x_n}$, and set
\[
\rho:=\Ind_{S_z}^G(\pi_{z}\rtimes V)\quad\text{and}\quad \rho_n:=\Ind_{S_{x_n}}^G(\pi_{x_n}\rtimes Q_n).
\]
 Suppose that $\rho_n\to \rho$ in $(C_0(X)\rtimes_\lt G)^\wedge$. Then $\ker \Ind_{S_{x_n}}^G Q_n\to \ker\Ind_{S_z}^G V$ in $\Id C^*(G)$.
\end{lemma}
\begin{proof}
Let $\overline{\rho}$ and $\overline{\rho}_n$ be the unique extensions of $\rho$ and $\rho_n$ (respectively) in $(M(C_0(X)\rtimes_\lt G))^{\wedge}$. Then $\overline{\rho}_n\to \overline{\rho}$.
Let $i_G: C^*(G)\to M(C_0(X)\rtimes_\lt G)$ be the canonical homomorphism. Then
$\overline{\rho}\circ i_G=\Ind^G_{S_z}V$ and $\overline{\rho}_n\circ i_G=\Ind^G_{S_{x_n}}Q_n$ for all $n$.
So it suffices to show that $\ker(\overline{\rho}_n\circ i_G) \to \ker(\overline{\rho}\circ i_G)$ in $\Id C^*(G)$.

Let $W$ be a basic open neighbourhood of $\ker(\overline{\rho}\circ i_G)$ in $\Id C^*(G)$. Then there exist $k\in\NN$ and $I_1,\ldots,I_k\in\Id C^*(G)$ such that $W =\{K\in\Id C^*(G): K\not\supset I_j \text{\, for\, } 1\leq j\leq k\}$. Let $K_j$ be the closed two-sided ideal of $M(C_0(X)\rtimes_\lt G)$ generated by $i_G(I_j)$. Since $\overline{\rho}(i_G(I_j))\neq\{0\}$ we have $\overline{\rho}(K_j)\neq\{0\}$.  Since $\overline{\rho}$ is irreducible $\ker\overline{\rho}$ is prime, and hence  $\overline{\rho}(K_1K_2\ldots K_k)\neq\{0\}$.
Since $\overline{\rho}_n\to \overline{\rho}$, there exists $N\in\NN$ such that $n\geq N$ implies
$\overline{\rho}_n(K_1K_2\ldots K_k)\neq\{0\}$.
Thus for $n\geq N$ and $1\leq j\leq k$, we have $\overline{\rho}_n(i_G(I_j))\neq\{0\}$
and hence $\ker(\overline{\rho}_n\circ i_G) \in W$.
\end{proof}

\begin{proof}[Proof of Theorem~\ref{thm-pss}] Set $\rho:=\Ind_{S_z}^G(\pi_{z}\rtimes V)$.

\medskip

\noindent Case~\eqref{thm-pss-item1}. Suppose that $S$ is a principal stability subgroup of $G$.  Then \[Y:=\{x\in X: S_x \mbox{ is conjugate to } S\}\] is a dense $G$-invariant subset of $X$. By Lemma~\ref{dense-in-spectrum},
\[\Ind Y:=\{ \Ind_{S_y}^G(\pi_{y}\rtimes U) :y\in Y, U\in \hat S_y \}\] is dense in $(C_0(X)\rtimes_\lt G)^\wedge$.

By \cite[Lemma~1.2]{AK} there exists a sequence $(\rho_n)$ in the dense subset $\Ind Y$ such that the upper multiplicity $M_U(\rho)$  equals the common upper and lower multiplicity $M(\rho, (\rho_n))$ of $\rho$ relative to the sequence $(\rho_n)$. We note again,  that if $(\rho_{n_k})$ is a subsequence of $(\rho_n)$, then  $M(\rho, (\rho_{n_k}))=M(\rho, (\rho_n))$ and that the multiplicity numbers depend only  on the class of the irreducible representation.  Since each $\rho_n\in\Ind Y$, there exist $x_n\in Y$ and $Q_n\in\hat S_{x_n}$ such that $\rho_n=\Ind_{S_{x_n}}^G(\pi_{x_n}\rtimes Q_n)$.

\begin{claimA*}
By passing to a subsequence and moving within orbits, we may assume that $x_n\to z$.
\end{claimA*}\label{ClaimA}
To see this, let $\Phi: (C_0(X)\rtimes_\lt G)^\wedge\to X/G$ be the function defined at \eqref{eq-Phi}. Then $q(x_n)=\Phi(\rho_n)\to \Phi(\rho)=q(z)$. Since $q$ is open, there is a subsequence $(x_{n_k})_{k\geq1}$ and a sequence $(g_k)$ in $G$
 such that $y_k:=g_k\cdot x_{n_k}\to z$.  Each $y_k$ is in $Y$ because $Y$ is $G$-invariant.  By
Lemma~\ref{lemma-D},
\[
\rho_{n_k}\simeq  \Ind_{S_{y_k}}^G(\pi_{y_k}\rtimes (g_k\cdot Q_{n_k})).
\]
Thus by replacing $x_{n_k}$ by $y_k$ and $Q_{n_k}$ by $g_k\cdot Q_{n_k}$, and relabeling, we may assume that $M_U(\rho)=M(\rho,(\rho_n))$, $x_n\in Y$ and $x_n\to z$. This completes the proof of Claim~A.

\begin{claimB*}
By passing to a subsequence and moving within orbits, we may assume that there is a closed subgroup $H$ of $G$ contained in $S_z$ such that $S_{x_n}=H$ for all $n\geq1$.
\end{claimB*}

Since $S$ is principal and each $x_n\in Y$, there exists $r_n\in G$ such that $S_{x_n}=r_n^{-1}S r_n$ for $n\geq1$. Since $G$ is compact, by passing to a subsequence we may assume that $r_n\to r$ in $G$. Then $u_n:=r^{-1} r_n\to e$ and  $y_n:=u_n\cdot x_n\to z$. Now
\[S_{y_n}=u_nS_{x_n}u_n^{-1}=r^{-1} r_nS_{x_n}r_n^{-1} r=r^{-1}Sr,\]
which is constant. Set $H=r^{-1}Sr$. For $h\in H$ we have $h\cdot y_n=y_n$ for all $n$, and hence $h\cdot z=z$. Thus $H\subset S_z$.
By  replacing $x_n\to z$ by $y_n\to z$ and invoking Lemma~\ref{lemma-D}, we may assume $S_{x_n}=H$. This completes the proof of Claim~B.

Let $R$ be any irreducible subrepresentation of $V|_H$.  By Proposition~\ref{ref-propn2}, applied with the constant sequence $(H,R)$, we obtain \begin{equation}\label{eq-max}M_U(\rho)\geq \big[ V|_H:R \big].\end{equation} It remains to show that there exists an irreducible subrepresentation $Q$ of $V|_H$ such that $M_U(\rho)=\big[ V|_H:Q \big]$.

\begin{claimC*} By passing to a subsequence, we may assume that there is an irreducible representation $Q$ of $H$ such that $Q_n=Q$ for all $n$.
\end{claimC*}
Since $\rho_n\to\rho$ in $(C_0(X)\rtimes_\lt G)^\wedge$ and $S_{x_n}=H$, we have $\ker\Ind_H^G Q_n\to \ker\Ind_{S_z}^G V$ in $\Id C^*(G)$ by Lemma~\ref{lem-park}.  Since $H\subset S_z$, induction in stages gives \[\Ind_H^G Q_n\simeq \Ind_{S_z}^G\big( \Ind_H^{S_z} Q_n \big).\]
Writing $V_n:=\Ind_H^{S_z} Q_n $, we have $\ker \Ind_{S_z}^G V_n\to \ker \Ind_{S_z}^G V$ in $\Id C^*(G)$.
Let $R$ be an irreducible sub-representation of $\Ind_{S_z}^G V$.  Thus $\ker \Ind_{S_z}^G V\subset\ker R$, and it follows that $\ker\Ind_{S_z}^G V_n\to \ker R$.
By \cite[Lemma 4.8]{EH}, there exists $N$ such that, for all $n\geq N$, $R$ is equivalent to a sub-representation of $\Ind_{S_z}^G V_n\simeq \Ind_H^G Q_n$.
Since $Q_n$ is irreducible, by Frobenius Reciprocity $Q_n$ is equivalent to a sub-representation of $R|_H$ for all $n\geq N$. But $R$ is an irreducible representation of the compact group $G$, hence is finite dimensional.  It follows that $R|_H$ has finite support in $\hat H$ and so there exists $[Q]$ in $\hat H$ such that the class $[Q_n]$ of $Q_n$ is frequently equal to $[Q]$. By Lemma~\ref{lem-induction}, $Q_n\simeq Q$ implies $\rho_n\simeq\Ind_H^G(\pi_{x_n}\rtimes Q)$. So by passing to a subsequence we may assume that $Q_n=Q$ for all $n$.  This completes the proof of Claim~C.

Now Proposition~\ref{prop-traces-converging-new-new} applies with the constant sequence $(S_{x_n}, Q_n)= (H,Q)$, and so for all positive $a$ in the dense subalgebra $C(G, C_c(X))$ of $C_0(X)\rtimes_\lt G$,
\begin{equation*}\label{eq-cor-traces2}
\tr(\rho_n(a))
\to
\sum _{W\in \hat S_z,\ [W|_H:Q]\neq 0}[W|_H:Q]\,\tr \big(\Ind_{S_z}^G(\pi_{z,W}\rtimes W) (a)\big).
\end{equation*}
By Lemma~\ref{lem-induction}, $\Ind_{S_z}^G(\pi_{z,W'}\rtimes W') \simeq  \Ind_{S_z}^G(\pi_{z,W}\rtimes W)$ implies $W\simeq  W'$. Thus
Theorem~4.1 of \cite{ASS} implies that the set of limits of $(\rho_n)$ is \[\{\Ind_{S_z}^G(\pi_{z,W}\rtimes W) :W\in \hat S_z, [W|_H:Q]\neq 0\}\]
and that for each such limit, $M\big(\Ind_{S_z}^G(\pi_{z,W}\rtimes W), (\rho_n)\big)=[W|_H:Q]$.
Applying this to the limit $\rho=\Ind_{S_z}^G(\pi_{z}\rtimes V) $ gives
 $M_U(\rho)=M(\rho,(\rho_n))=[V|_H:Q]$ as required. Combining this with \eqref{eq-max} gives  $M_U(\rho)=\max_{R\in\hat H}\big[ V|_H:R \big]$. This completes the proof of case \eqref{thm-pss-item1}.

\medskip

\noindent Case~\eqref{thm-pss-item2}.  Suppose that $X$ is locally of finite $G$-orbit type. Then there exist a $G$-invariant open neighbourhood $U$ of $z$ in $X$, $k\in \NN$ and closed subgroups $H_1, \ldots, H_k$ of $G$ such that, for all $x\in U$, $S_x$ is conjugate to one of the $H_j$.

 By \cite[Lemma~1.2]{AK}, there exists a sequence $\rho_n\in (C_0(X)\rtimes_\lt G)^{\wedge}$ such that $M_U(\rho)=M(\rho, (\rho_n))$.  Arguing as in Claim~A above, we may assume that $\rho_n=\Ind_{S_{x_n}}^G(\pi_{x_n}\rtimes Q_n)$ where $Q_n\in {\hat S}_{x_n}$  for $n\geq 1$ and $x_n\to z$ in $X$. (In case~\eqref{thm-pss-item1} we had the $\rho_n$ in a distinguished dense subset of $(C_0(X)\rtimes_\lt G)^{\wedge}$ and needed to keep them there under the Claim~A manoeuvre - this is the only difference.)

Eventually $x_n\in U$ and  there exists $j\in\{1,\ldots,k\}$ such that $S_{x_n}$ is frequently conjugate to $H_j$. By passing to a subsequence, we may assume that there exists $j$ such that $S_{x_n}$ is conjugate to $H_j$ for $n\geq 1$. By passing to a further subsequence and moving within orbits as in Claim~B above, we may assume that there is a closed subgroup $H$ of $G$ contained in $S_z$ such that $S_{x_n}=H$ for all $n\geq1$.  Let $R$ be any irreducible sub-representation of $V|_H$. By Proposition~\ref{ref-propn2}, applied with the constant sequence $(H,R)$, we obtain
\[
M_U(\rho)\geq \big[ V|_H:R \big]
\]
as in~\eqref{eq-max} in  case~\eqref{thm-pss-item1}. The remainder of the proof follows as in case~\eqref{thm-pss-item1}.
\end{proof}

\begin{cor}\label{cor-1}  Let $(G,X)$ be a second-countable transformation group with $G$ compact and with a principal stability group $S$. Let $z\in X$, suppose that $S_z$ is conjugate to $S$ and let $V\in \hat S_z$.  Then \[M_U\big(\Ind_{S_z}^G(\pi_{z}\rtimes V)\big)=1.\]
\end{cor}

\begin{proof} By Theorem~\ref{thm-pss} there exist a closed subgroup $H$ of $G$ such that $H$ is conjugate to $S$ and $H\subset S_z$, and  $Q\in \hat H$ such that $M_U\big(\Ind_{S_z}^G(\pi_{z}\rtimes V)\big)=[V|_H:Q]$. In particular, $[V|_H:Q]\geq 1$. But $S_z$ is also conjugate to $S$. Hence  $H\subset S_z=gHg^{-1}$ for some $g\in G$. Since $G$ is compact, $H=gHg^{-1}$ \cite[Proposition~1.9]{Bredon}. Thus $H=S_z$. Now $[V|_H:Q]=[V:Q]=1$ since $V$ is irreducible.
\end{proof}

The next corollary applies, in particular, in the case where the trivial subgroup is a principal stability subgroup of $G$ (see Examples ~\ref{ex-S3} and ~\ref{ex-D4}).

\begin{cor}\label{new-cor-pss-trivial} Let $(G,X)$ be a second-countable transformation group with $G$ compact. Suppose that $G$ has a principal stability subgroup $S$ contained in the centre $Z(G)$ of $G$.
\begin{enumerate}
\item Let $z\in X$ and $V\in\hat S_z$. Then $M_U\big(\Ind_{S_z}^G(\pi_{z}\rtimes V)\big)= \dim V$.
\item $C_0(X)\rtimes_\lt G$ is a Fell algebra if and only if $S_z$ is abelian for every $z\in X$.
\end{enumerate}
\end{cor}

\begin{proof}
Let $z\in X$ and $V\in {\hat S_z}$.
By Theorem~\ref{thm-pss} there exists a closed subgroup $H$ of $G$, conjugate to $S$, such that $H\subset S_z$ and
\[M_U\big(\Ind_{S_z}^G(\pi_{z}\rtimes V)\big)=\max_{R\in\hat H}\big[ V|_H:R \big].\]
Since $S\subset Z(G)$, we have $H=S$.
But $V(S)$ is contained in the commutant $\C1_{\H_V}$ of $V(S_z)$ and so $V|_S = (\dim V)\tau$ for some $\tau\in \hat S$. This gives (1).

It now follows from  \cite[Theorem~4.6]{A} that $C_0(X)\rtimes_\lt G$ is a Fell algebra if and only if $\dim V=1$ for all such $z$ and $V$. Since $S_z$ is abelian if and only if every irreducible representation is $1$-dimensional, (2) follows.
\end{proof}

\begin{example}\label{ex-orthogonal}
Let $G$ be the special orthogonal group $\SO(n)$ of orthogonal matrices in $M_n(\R)$ with determinant $1$.  Then $G$ is compact, and  $G$ acts on $X=\R^n$ by matrix multiplication $A\cdot x=Ax$.  We will show, using Theorem~\ref{thm-pss},  that $C_0(X)\rtimes_\lt G$ is a Fell algebra.

Since each element of $G$ is an isometry, the orbits are parametrised by $r\in[0,\infty)$ with $O_r=\{x\in\R^n: \|x\|=r\}$.   The orbit $O_0$ is $\{0\}$ and the stability subgroup at $0$ is $S_0=G$.  Fix $r>0$. Take $x_r:=r(1,0,\dots, 0)$ to represent the orbit $O_r$. Then $S_{x_r}=1\oplus\SO(n-1)$, the subgroup of $G$ leaving the first coordinate fixed.  If $y\in O_r$ then the stability subgroup $S_y$ is conjugate to $S_{x_r}$.  Thus $1\oplus\SO(n-1)$ is a principal stability subgroup with dense subset
\[
Y:=\{x\in X: \text{$S_x$ is conjugate to $1\oplus\SO(n-1)$}\}=\R^n\setminus\{0\}\subset X.
\]
By Corollary~\ref{cor-1}  we have  $M_U\big(\Ind_{S_y}^G(\pi_{y}\rtimes V)\big)=1$ for all $y\in Y$ and $V\in \hat S_y$.

However, for these representations, more is already known. Since $Y$ is open and $G$-invariant, $C_0(Y)\rtimes_\lt G$ is an ideal in $C_0(X)\rtimes_\lt G$. The stabiliser map $y\mapsto S_y$ is continuous on $Y$ by, for example, \cite[Lemma~1.2]{AK-stablerank}. Thus $C_0(Y)\rtimes_\lt G$ has continuous trace by \cite[Corollary~2]{E}. Thus $M_U\big(\Ind_{C_0(Y)\rtimes S_y}^{C_0(Y)\rtimes G}(\pi_{y}\rtimes V)\big)=1$ for all $y\in Y$ and $V\in \hat S_y$.  By \cite[Lemma~2.7]{ASS}, the corresponding irreducible representation $\Ind_{S_y}^G(\pi_{y}\rtimes V)$ of $C_0(X)\rtimes_{\lt}G$ also has upper multiplicity $1$, as we have already seen above.

It remains to consider $\rho_{0,V}:=\Ind_{S_0}^G(\pi_{0}\rtimes V)=\pi_{0}\rtimes V$ where $V\in \hat S_0=\hat G$.  By Theorem~\ref{thm-pss} there exists a closed subgroup $H$ of $G$ such that $H$ is conjugate to $1\oplus \SO(n-1)$ and a representation $R\in\hat H$ such that $M_U(\rho_{0,V})=\big[V|_H:R\big]$. Since $H$ is conjugate to $1\oplus \SO(n-1)$ there exists $g\in G$ such that $g\cdot R$ is an irreducible sub-representation of $(g\cdot V)|_{1\oplus \SO(n-1)}$ and  $M_U(\rho_{0,V})=\big[(g\cdot V)|_{1\oplus \SO(n-1)}:g\cdot R\big]$. By the branching theorem for $\SO(n)$ with respect to $\SO(n-1)$ (see \cite[\S1 and \S3]{Knapp} and \cite{Mur}), each irreducible sub-representation of $(g\cdot V)|_{1\oplus \SO(n-1)}$ occurs with multiplicity $1$. It follows that $M_U(\rho_{0,V})=1$.

We have shown that $M_U(\rho)=1$ for all $\rho\in (C_0(X)\rtimes_\lt G)^\wedge$. Thus $C_0(X)\rtimes_\lt G$ is a Fell algebra by \cite[Theorem~4.6]{A}.
\end{example}

\begin{example}\label{ex-S3}  The symmetric group $G=S_3$ acts on $X=\R^3$ by $\sigma\cdot(x_1, x_2, x_3)=(x_{\sigma(1)}, x_{\sigma(2)}, x_{\sigma(3)})$.  Here $\{e\}$ is a principal stability subgroup with
\[
Y=\{x\in\R^3: \text{$S_x$ is conjugate to $\{e\}$}\}=\{x\in\R^3: x_i\neq x_j \text{\ if\ }i\neq j\}.
\]
By Corollary~\ref{cor-1}, $M_U\big(\Ind_{S_y}^G(\pi_{y}\rtimes 1|_{S_y})\big)=1$ for all $y\in Y$.
(Again $Y$ is open in $X$, and we note that, moreover, $C_0(Y)\rtimes_\lt G$ has continuous trace by \cite[Theorem~17]{green} because the action of $G$ on $Y$ is free.)

Let $z\in X\setminus Y$.  Then the stability subgroup $S_z$ is one of $\{e, (12)\}$, $\{e, (13)\}$, $\{e, (23)\}$ or $G$. If $S_z$ has order $2$ then every irreducible representation of $S_z$ is $1$-dimensional, and $M_U(\Ind_{S_z}^G(\pi_{z}\rtimes V))=1$ for every $V\in \hat S_z$ by Corollary~\ref{cor-dim1}. If $S_z=G$ then $\hat S_z=\{1, \sign, Q\}$ where $Q$ is the $2$-dimensional representation of $S_3$ associated to the partition $(2,1)$ of $3$ (see, for example, \cite[Theorem 4.12 and Example 5.1]{James}).
By Corollary~\ref{new-cor-pss-trivial}(1),
\[
M_U(\Ind_{S_z}^G(\pi_{z}\rtimes V))=
\begin{cases}1 &\text{if $V=1$}\\
1&\text{if $V=\sign$}\\
2 &\text{if $V=Q$}.
\end{cases}
\]
In particular, $C_0(X)\rtimes_\lt G$ is not a Fell algebra.
\end{example}

\begin{example}\label{ex-D4}
We consider \cite[Example 3.5]{EH} in which $G=D_4$, the $8$-element dihedral group, acts on $X=\T^2$.
The trivial subgroup is a principal stability group and the only non-abelian stability subgroup is $D_4$ itself, occurring at the points $(1,1)$ and $(-1,-1)$ in $\T^2$. Let $\lambda$ be the `standard' $2$-dimensional irreducible representation of $D_4$. It follows from Corollary~\ref{new-cor-pss-trivial}(1) that the only non-Fell points in the spectrum of $C(X)\rtimes_\lt G$ are $\pi_{(1,1)}\rtimes \lambda$ and $\pi_{(-1,-1)}\rtimes \lambda$, and that both have upper multiplicity equal to $2$.
\end{example}

\begin{remark}\label{remark-Gfinite}
Motivated by the previous two examples, we briefly consider the case of a second countable transformation group $(G,X)$ in which $G$ is finite. Although there need not be a principal stability group, the space $X$ is automatically of finite $G$-orbit type. By either Theorem~\ref{thm-new-general} or Theorem~\ref{thm-pss}, for $z\in X$ and $V\in \hat S_z$, there is a closed subgroup $H$ of $S_z$ and $R\in \hat H$ such that
$$M_U\big(\Ind_{S_z}^G(\pi_{z}\rtimes V)\big)=\big[ V|_H:R \big] = \big[\Ind^{S_z}_H R:V \big],$$
where the second equality follows from Frobenius Reciprocity. Thus
$$M_U\big(\Ind_{S_z}^G(\pi_{z}\rtimes V)\big)(\dim R)\leq \dim V$$
 and
$$M_U\big(\Ind_{S_z}^G(\pi_{z}\rtimes V)\big)(\dim V) \leq \dim (\Ind^{S_z}_H R) = (\dim R)[S_z:H].$$
It follows from these inequalities that $M_U\big(\Ind_{S_z}^G(\pi_{z}\rtimes V)\big)^2 \leq [S_z:H]$. If there does exist a principal stability group $S$ then $H$ may be chosen conjugate to $S$ (Theorem~\ref{thm-pss}) and hence
$$M_U\big(\Ind_{S_z}^G(\pi_{z}\rtimes V)\big)^2 \leq \frac{|S_z|}{|S|} \leq [G:S].$$
The final inequality is somewhat similar to the estimate for the upper multiplicity of irreducible representations of Moore groups in \cite[Corollary 2.3]{AK}.
\end{remark}

\section{Open subsets of Fell points in the spectrum}\label{sec-Fellpoints}

Let $(G, X)$  be a second-countable transformation group with $G$ compact.
In this section we consider interesting open subsets  of irreducible representations in $(C_0(X)\rtimes_\lt G)^\wedge$ satisfying Fell's condition.  For example, the  set we consider in Lemma~\ref{lem-faset} (see also \cite[Theorem~3.3]{MR} and  \cite[\S3]{EH}) is homeomorphic to the spectrum of the fixed-point algebra.

We start by recalling some background. Let $G$ be a compact group and $\alpha: G\to\Aut A$ be a continuous action of $G$ by automorphisms of a $C^*$-algebra $A$.  Define $p:G\to M(A)$ by $p(t)=1$ for all $t\in G$.  As pointed out in \cite{R}, $pL^1(G,A)p$ is the closed subalgebra of $L^1(G,A)$ consisting of constant functions from $G$ into the fixed-point algebra  \[A^\alpha:=\{a\in A :\alpha_s(a)=a\text{\ for all $s\in G$}\}.\]  Thus $f\mapsto f(e)$ is an isomorphism of $pL^1(G,A)p$ onto $A^\alpha$. Passing to the completion, we get $p\in M(A\rtimes_\alpha G)$ such that the hereditary subalgebra $p(A\rtimes_\alpha G)p$ of $A\rtimes_\alpha G$ is isomorphic to  $A^\alpha$.  It follows that the ideal
\[I_{\GFA}:=\overline{(A\rtimes_\alpha G)p(A\rtimes_\alpha G)}\]
is Morita equivalent to $A^\alpha$ via the imprimitivity bimodule $(A\rtimes_\alpha G)p$. An action of a compact group on $A$ is in particular a proper action on $A$ in the sense of \cite[Definition~1.2]{Rieffel}.  When $G$ is compact, the Morita equivalence built in \cite{Rieffel} reduces to the  $I_{\GFA}$--$A^\alpha$ Morita equivalence of \cite{R} discussed above.

The $I_{\GFA}$--$A^\alpha$ Morita equivalence  has been exploited widely.  For example, Gootman and Lazar use non-abelian duality in \cite[Theorem~3.2]{GL2} to prove that for the action  of a compact group on $A$, the crossed product $A\rtimes_\alpha G$ is liminal (postliminal) if and only if the fixed-point algebra $A^\alpha$ is liminal (postliminal).  The ``if'' direction  of this sort of result fails for Fell algebras, as the following example shows.

\begin{example} Let $(G,X)=(S_3,\R^3)$ be the transformation group of Example~\ref{ex-S3}. We showed there that $C_0(X)\rtimes_\lt G$ is not a Fell algebra. But the fixed-point algebra $C_0(X)^\lt$, being commutative, is a Fell algebra.
 \end{example}

\begin{lemma}\label{lem-faset}  Let $(G,X)$ be a second-countable transformation group with $G$ compact.  Then the ideal $I_{\GFA}$ is a $C^*$-algebra with continuous trace  and spectrum homeomorphic to
\[
\{ \Ind_{S_x}^G(\pi_{x}\rtimes 1|_{S_x}):x\in X\}.
\]
\end{lemma}

\begin{proof} Since $I_{\GFA}$ is Morita equivalent to $C_0(X)^\lt$, which is commutative, $I_{\GFA}$ has continuous trace. Theorem~3.3 of \cite{MR} says that
\[I_{\GFA}=\bigcap \{ \ker\Ind_{S_x}^G(\pi_{x}\rtimes V):x\in X, V\in \hat S_x, V\neq 1\}.\]
But the proof of \cite[Theorem~3.3]{MR} shows slightly more:  that $I_{\GFA}\subset \ker\Ind_{S_x}^G(\pi_{x}\rtimes V)$ if and only if $V\neq 1$.  Thus $\{\Ind_{S_x}^G(\pi_{x}\rtimes V):x\in X, V\in \hat S_x, V\neq 1\}$ is closed in $(C_0(X)\rtimes_\lt G)^\wedge$, and the lemma follows.
\end{proof}

\begin{lemma}\label{lem-tau} Let $(G,X)$ be a second-countable transformation group with $G$ compact and $\tau$ a character of $G$. Let $z, x_n\in X$ and $V_n\in \hat S_{x_n}$ for $n\geq 1$. Suppose that $\Ind_{S_{x_n}}^G(\pi_{x_n}\rtimes V_n)\to \Ind_{S_z}^G(\pi_{z}\rtimes \tau|_{S_z})$ in $(C_0(X)\rtimes_\lt G)^\wedge$.  Then  $V_n=\tau|_{S_{x_n}}$ eventually.
\end{lemma}

\begin{proof}
By Lemma~\ref{lem-park},
\[\ker\Ind_{S_{x_n}}^G V_n \to \ker \Ind_{S_z}^G(\tau|_{S_z})\]
 in $\Id C^*(G)$.
Since $G$ is amenable, we have $\ker \Ind_{S_z}^G(\tau|_{S_z})\subset \ker \tau$ by \cite[Theorem~6.9]{KT} and hence $\ker\Ind_{S_{x_n}}^G V_n\to \ker\tau$ in $\Id C^*(G)$.  By \cite[Lemma 4.8]{EH}, there exists $N$ such that, for all $n\geq N$, $\tau$ is equivalent to a sub-representation of $\Ind_{S_{x_n}}^G V_n$. By Frobenius Reciprocity, $V_n$ is equivalent to a sub-representation of the $1$-dimensional representation $\tau|_{S_{x_n}}$ and hence $V_n=\tau|_{S_{x_n}}$ when $n\geq N$.
\end{proof}

\begin{prop}\label{prop-openset}
Let $(G,X)$ be a second-countable transformation group with $G$ compact and $\tau$ a character of $G$.  Then
\[
O_\tau:=\{ \Ind_{S_x}^G(\pi_{x}\rtimes V):x\in X, V=\tau|_{S_x}\}
\]
is open in $(C_0(X)\rtimes_\lt G)^\wedge$.
\end{prop}

\begin{proof} Fix $\rho:=\Ind_{S_z}^G(\pi_{z}\rtimes \tau|_{S_z})$. Suppose, by way of contradiction, that there is no open subset $U$ of $(C_0(X)\rtimes_\lt G)^\wedge$ such that $\rho\in U\subset O_\tau$.  Let $(U_n)$ be a decreasing sequence of basic open neighborhoods of $\rho$. Then, for every $n$, $U_n\not\subset O_\tau$, and so there exist $x_n\in X$ and  $V_n\in {\hat S}_{x_n}\setminus \{ \tau|_{S_{x_n}}\}$ such that $\Ind_{S_{x_n}}^G(\pi_{x_n}\rtimes V_n)\in U_n$. Then $\Ind_{S_{x_n}}^G(\pi_{x_n}\rtimes V_n)\to \rho$.
By Lemma~\ref{lem-tau}, $V_n=\tau|_{S_{x_n}}$ eventually, a contradiction. It follows that $O_\tau$ is open.
\end{proof}

\begin{cor}\label{cor-open-Fell}
 Let $(G,X)$ be a second-countable transformation group with $G$ compact. Then the subset
 \begin{equation*}\label{openset}
O_{\characters}:=\{\Ind_{S_x}^G(\pi_{x}\rtimes V):x\in X, V=\tau|_{S_x} \text{\ for some  character $\tau$ of $G$}\}
 \end{equation*}
of $(C_0(X)\rtimes_\lt G)^\wedge$ is open and the corresponding ideal
 of $C_0(X)\rtimes_\lt G$ is a Fell algebra.
\end{cor}

\begin{proof} By Proposition~\ref{prop-openset}, $O_{\characters}$ is a union of open sets and hence is open.  Since $V$ is a character of $S_z$, each $\Ind_{S_x}^G(\pi_{x}\rtimes V)$ in $O_{\characters}$ has upper multiplicity $1$ by Corollary~\ref {cor-dim1}. Thus the ideal corresponding to $O_{\characters}$ is a Fell algebra.
\end{proof}

\begin{remark}
Suppose that $G$ is abelian. Then every irreducible representation of $G$ is a character, and  every character of a stability subgroup extends to a character of $G$ (see, for example, \cite[Corollary~24.12]{Hewitt-Ross}).  Thus $O_{\characters}$ is all of the spectrum of $C_0(X)\rtimes G$.
\end{remark}

\begin{remark}  The set $O_{\characters}$ is not necessarily Hausdorff. To see this, let $(G,X)=(S_3,\R^3)$ be the transformation group of Example~\ref{ex-S3}. Let $x_n=(0,2/n,1/n)$ and $z=(0,0,0)$. Then $x_n\to z$,  $S_{x_n}=\{e\}$ and $S_z=S_3$. Set $\rho_n=\Ind_{\{e\}}^G(\pi_{x_n}\rtimes 1|_{\{e\}})$.   Apply Proposition~\ref{prop-traces-converging-new-new}  with the trivial sequence $(\{e\}, 1))_n$ to get
\[
\tr\big( \rho_n(a) \big)
\to
\sum_{W\in \hat S_3}\big[W|_{\{e\}}:1|_{\{e\}}\big]\tr\big(\pi_{z,W}\rtimes W)(a) \big)
\]
for all $a\in C(G, C_c(X))^+$.  Note that $\big[W|_{\{e\}}:1|_{\{e\}}\big]\geq 1$ for all $W\in \hat S_3$.
By \cite[Theorem~4.1]{ASS}, the set of limits of $(\rho_n)$ is $\{\pi_{z,W}\rtimes W:W\in\hat S_3\}$. In particular, $(\rho_n)$ is a sequence in $O_{\characters}$ converging to the distinct points  $\pi_{z}\rtimes 1$ and $\pi_{z}\rtimes \sign$  in $O_{\characters}$. Thus $O_{\characters}$ is not Hausdorff.
\end{remark}

\section{Extension to Cartan transformation groups}\label{sec-Cartan}

In this section, we consider a second countable transformation group $(G,X)$ in which the group $G$ is not necessarily compact. A subset $Y$ of $X$ is \emph{wandering} if
\[
\{s\in G: s\cdot Y\cap Y\neq \emptyset\}
\]
is a relatively compact subset of $G$ (such subsets $Y$ are called ``thin'' in \cite[Definition~1.1.1]{palais} and ``wandering'' in, for example, \cite[Definition on p.88]{green1}). The transformation group $(G,X)$ is \emph{proper} (or $G$ \emph{acts properly on} $X$) if every compact subset of $X$ is wandering. Equivalently, $(G,X)$ is proper if the function $(g,x)\mapsto (g\cdot x, x)$ from $G\times X$ to $X\times X$ is proper in the sense that inverse images of compact sets are compact. More generally, $(G,X)$ is said to be \emph{Cartan} (or $X$ is a \emph{Cartan $G$-space}) if every element of $X$ has a wandering neighbourhood. Indeed, $(G,X)$ is Cartan if and only if it is proper and the orbit space $X/G$ is Hausdorff \cite[Theorem~1.2.9]{palais}. The example on page~298 of \cite{palais}, in which the Hausdorff space $X$ is indeed locally compact, shows that a  Cartan transformation group may not be proper.
 If $(G,X)$ is Cartan, then all the stability subgroups are compact and the orbits are closed \cite[Proposition~1.1.4]{palais}.  Moreover, if $U$ is a non-empty wandering open subset of $X$, then $(G, G\cdot U)$ is proper by  \cite[Proposition~1.2.4]{palais}. Hence
the Cartan transformation groups $(G,X)$ are precisely those which are \emph{locally proper} in the sense that each point of $X$ has a $G$-invariant open neighbourhood on which $G$ acts properly.

Proper transformation groups  are locally induced from  compact transformation groups  \cite{abel}, and we now discuss how this can be used to extend some of our results to the more general setting of Cartan transformation groups. We thank the referee for pointing this out to us.

Let $(G, X)$ be a Cartan transformation group. Let $z\in X$ and let $U_0$ be a $G$-invariant neighborhood of $z$ such that $(G, U_0)$ is proper. By  \cite[Theorem~3.3 ]{abel} there exist an open  $G$-invariant neighbourhood $U$ of $z$ in $U_0$, a compact subgroup $K$ of $G$ with $S_z\subset K$ and a $G$-equivariant map $f:U\to G/K$ with $f(z)=K$.  Then the closed $K$-invariant subset $Y: = f^{-1}(K)$ of $U$ contains $z$ and has the property that $U$ is $G$-equivariantly homeomorphic to an induced space $(G\times Y)/K$ \cite[Lemma~3.5]{abel}. It follows from the $G$-equivariance of $f$ that, for $y\in Y$, the stability subgroup in $G$ is the same as in $K$.

Since  $Y$ is closed in the locally compact space $U$ it is  locally compact.  It then follows from a special case of Raeburn's Symmetric Imprimitivity Theorem (\cite{rae-sit}, see \cite[Corollary~4.17]{tfb^2}), that   $C_0(U)\rtimes_{\lt} G$  and $C_0(Y)\rtimes_{\lt} K$ are Morita equivalent.  We refer to  \cite[\S1]{EH}  and the preamble to Proposition~3.13  of \cite{EH} for the details.   We can now show how to generalise Theorem~\ref{thm-new-general} from  compact to Cartan transformation groups.

\begin{thm}\label{thm-Cartan} Let $(G,X)$ be a second-countable Cartan transformation group. Let $z\in X$ and $V\in \hat S_z$.
 Then there exist  a closed subgroup $H$ of $G$ with $H\subset S_z$ and an irreducible subrepresentation $R$ of $V|_H$
such that
\[M_U\big(\Ind_{S_z}^G(\pi_{z}\rtimes V)\big) =[ V|_H:R \big] \leq \dim V < \infty.\]
Moreover, if the stabiliser map $X\to \Sigma(G)$, $x\to S_x$, is continuous at $z$ then
$$M_U\big(\Ind_{S_z}^G(\pi_{z}\rtimes V)\big) =1.$$
\end{thm}

\begin{proof}
By the discussion above, there exist an open $G$-invariant  neighbourhood $U$ of $z$ in $X$, a compact subgroup $K$ of $G$ containing $S_z$, a $K$-invariant closed subset $Y$ of $U$ containing $z$ such that $S_y\subset K$ for all $y\in Y$, and  an $\big(C_0(U)\rtimes_{\lt} G\big)$--$\big(C_0(Y)\rtimes_{\lt} K\big)$ imprimitivity bimodule $Z$ arising from the Symmetric Imprimitivity Theorem.

By \cite[Lemma~2.7]{ASS}, the upper multiplicity of $\Ind_{C_0(X)\rtimes S_z}^{C_0(X)\rtimes G}(\pi_{z}\rtimes V)$ can be computed in the spectrum of the ideal $C_0(U)\rtimes_\lt G$ of $C_0(X)\rtimes_\lt G$.  By \cite[Proposition~3.13]{EH},  \[Z\!-\!\!\Ind\big(\Ind^{C_0(Y)\rtimes K}_{C_0(Y)\rtimes S_z}(\pi_z\rtimes V)\big) \text{\ \ and \ \ }\Ind^{C_0(U)\rtimes G}_{C_0(U)\rtimes S_z}(\pi_z\rtimes V)\]
are unitarily equivalent, and by \cite[Corollary~13]{aHRW} the upper multiplicities of  \[Z\!-\!\!\Ind\big(\Ind^{C_0(Y)\rtimes K}_{C_0(Y)\rtimes S_z}(\pi_z\rtimes V)\big)
\text{\ \ and \ \ }
\Ind^{C_0(Y)\rtimes K}_{C_0(Y)\rtimes S_z}(\pi_z\rtimes V)\]   coincide.   Theorem~\ref{thm-new-general} applied to the second-countable  transformation group $(K, Y)$ gives a closed subgroup $H$ of $K$ contained in $S_z$ and an irreducible subrepresentation $R$ of $V|_H$ such that
\[M_U\big(\Ind_{C_0(X)\rtimes S_z}^{C_0(X)\rtimes G}(\pi_{z}\rtimes V)\big)=M_U\big(\Ind^{C_0(Y)\rtimes K}_{C_0(Y)\rtimes S_z}(\pi_z\rtimes V)) =[ V|_H:R \big] \leq \dim V < \infty.\]

Finally, suppose that the stabiliser map $X\to \Sigma(G)$ is continuous at $z$. Then the stabiliser map $Y\to \Sigma(K)$ is also continuous at $z$ and hence
\[M_U\big(\Ind_{C_0(X)\rtimes S_z}^{C_0(X)\rtimes G}(\pi_{z}\rtimes V)\big)=M_U\big(\Ind^{C_0(Y)\rtimes K}_{C_0(Y)\rtimes S_z}(\pi_z\rtimes V)) =1\] by Theorem~\ref{thm-new-general}. \qedhere
\end{proof}

By using Theorem~\ref{thm-Cartan} in place of Theorem~\ref{thm-new-general}, the proof of Corollary~\ref{cor-dim1} extends to give the following result. Part~\eqref{cor-characterFell-item2} improves \cite[Lemma~5.10]{aH} which dealt with compact abelian groups $G$.

\begin{cor} Let $(G,X)$ be a second-countable Cartan transformation group.
\begin{enumerate}
\item Let $z\in X$ and $V\in {\hat S_z}$ be a character.  Then \[M_U\big(\Ind_{S_z}^G(\pi_{z}\rtimes V)\big)=1.\]
\item\label{cor-characterFell-item2} Suppose that $S_z$ is abelian for all $z\in X$. Then $C_0(X)\rtimes_{\rm lt}G$ is a Fell algebra.
\end{enumerate}
\end{cor}

The next result is an immediate corollary of Theorem~\ref{thm-Cartan} and extends part of an earlier result for free actions \cite[Theorem]{aH-Fell}. Note that Example~\ref{ex-orthogonal} shows that the converse is false.

\begin{cor} Let $(G,X)$ be a second-countable Cartan transformation group.
If the stabiliser map $X\to \Sigma(G)$, $x\to S_x$, is continuous then $C_0(X)\rtimes_{\rm lt}G$ is a Fell algebra.
\end{cor}

Finally, it seems worthwhile to record explicitly the generalisation of \cite[Corollary 2]{E} from actions of compact groups to proper actions, which may well be known to some experts. Indeed, the equivalence of conditions \eqref{item1-genS} and \eqref{item2-genS}  is shown in \cite[Theorem 5.5.1]{Neu}.

\begin{thm}
Let $(G,X)$ be a proper transformation group. The following conditions are equivalent.
\begin{enumerate}
\item\label{item1-genS} The stabiliser map $X\to \Sigma(G)$, $x\to S_x$, is continuous.
\item\label{item2-genS}  The spectrum of $C_0(X)\rtimes_{\rm lt}G$ is Hausdorff.
\item\label{item3-genS} $C_0(X)\rtimes_{\rm lt}G$ has continuous trace.
\end{enumerate}
\end{thm}

\begin{proof}
Assuming \eqref{item1-genS} (and hence \eqref{item2-genS}), let $z\in X$ and $V\in {\hat S_z}$. To see \eqref{item3-genS}, it suffices to show that $\Ind_{S_z}^{G}(\pi_{z}\rtimes V)$ satisfies Fell's condition. In the second countable case, this follows from Theorem~\ref{thm-Cartan}. In general, using the notation of this section, it follows from \cite[Corollary 2]{E} that $C_0(Y)\rtimes_{\rm lt}K$ has continuous trace. Then $C_0(U)\rtimes_{\rm lt} G$ also has continuous trace since this property is preserved by Morita equivalence \cite[Proposition 7]{aHRW}. In particular, the restriction of $\Ind_{S_z}^{G}(\pi_{z}\rtimes V)$ to $C_0(U)\rtimes_{\rm lt} G$ satisfies Fell's condition and hence so does $\Ind_{S_z}^{G}(\pi_{z}\rtimes V)$ itself. That \eqref{item3-genS} implies \eqref{item2-genS} follows since all  $C^*$-algebras with continuous trace have Hausdorff spectrum.
\end{proof}



\begin{thebibliography}{99}
%

\bibitem{abel} H. Abels, \emph{A universal proper $G$-space}, Math. Z. \textbf{159} (1978), 143--158.


\bibitem{A}  R.J. Archbold,  \emph{Upper and lower multiplicity for irreducible representations of $C^*$-algebras}, Proc. London Math. Soc. \textbf{69} (1994), 121--143.

\bibitem{Aprimals} R.J. Archbold, \emph{Topologies for primal ideals}, J. London Math. Soc. (2) 36 (1987) 524--542.


\bibitem{AK} R.J. Archbold and E. Kaniuth, \emph{Upper and lower multiplicity for irreducible representations of SIN-groups}, Illinois J. Math. \textbf{43} (1999), 692--706.

\bibitem{AK-stablerank}  R.J. Archbold and E. Kaniuth, \emph{Stable rank and real rank of compact transformation group $C^*$-algebras}, Studia Math. \textbf{175} (2006), 103--120.


\bibitem{ASo} R.J. Archbold and D.W.B. Somerset, \emph{Transition probabilities and trace functions for {$C^*$}-algebras}, Math. Scand. \textbf{73} (1993), 81--111.

\bibitem{AS} R.J. Archbold and J.S. Spielberg, \emph{Upper and lower multiplicity for irreducible representations of $C^*$-algebras. \textrm{II}}, J. Operator Theory \textbf{36} (1996), 201--231.
%
\bibitem{ASS} R.J. Archbold, D.W.B. Somerset and J.S. Spielberg, \emph{Upper multiplicity and bounded trace ideals in $C^*$-algebras}, J. Funct. Anal. \textbf{146} (1997), 430--463.

\bibitem{Bag} L. Baggett, \emph{A description of the topology on the dual spaces of certain locally compact groups}, Trans. Amer. Math. Soc. \textbf{132} (1968), 175--215.



\bibitem{duflo} P.~Bernat, N.~Conze, M.~Duflo, M.~L{\'e}vy-Nahas, M.~Ra{\"\i}s, P.~Renouard and M.~Vergne, Repr\'esentations des Groupes de Lie R\'esolubles, ch.~V, pp.~93--119, Dunod, 1972, Monographies de la Soci\'et\'e Math\'ematique de France, No. 4.

\bibitem{Bredon} G.E. Bredon, Introduction to Compact Transformation Groups, Academic Press, 1972.


%
\bibitem{DE} A. Deitmar and S. Echterhoff, Principles of Harmonic Analysis, Springer,  2009.


\bibitem{Dix2} J. Dixmier, $C^*$-Algebras, North Holland, 1977.


\bibitem{E} S. Echterhoff, \emph{On transformation group $C^*$-algebras with continuous trace}, Trans. Amer. Math. Soc. \textbf{343} (1994), 117--133.

\bibitem{EH} S. Echterhoff and H. Emerson, \emph{Structure and K-theory of crossed products by proper actions}, Expo. Math. \textbf{29} (2011) 300--344.





\bibitem{Fell-topology}
J.M.G. Fell, \emph{A Hausdorff topology for the closed subsets of a locally
  compact non-Hausdorff space}, Proc. Amer. Math. Soc. \textbf{13} (1962),
  472--476.

\bibitem{FellII} J.M.G. Fell,  \emph{Weak containment and induced representations of groups. II}, Trans. Amer. Math. Soc. \textbf{110} (1964), 424--447.



\bibitem{GL}
E.C. Gootman and A.J. Lazar, \emph{Applications of non-commutative duality to
crossed product {$C^*$}-algebras determined by an action or coaction}, Proc.
London Math. Soc. \textbf{59} (1989), 593--624.


\bibitem{GL2}
E.C. Gootman and A.J. Lazar, \emph{Compact group actions on {$C^*$}-algebras: an application of non-commutative duality}, J. Funct. Anal. \textbf{91} (1990), 237--245.

\bibitem{green1} P. Green, \emph{$C^*$-algebras of transformation groups with smooth orbit space}, Pacific J. Math. \textbf{72} (1977), 71--97.


%
\bibitem{green}
P. Green, \emph{The local structure of twisted covariance algebras}, Acta Math.
\textbf{140} (1978), 191--250.


\bibitem{Hewitt-Ross} E. Hewitt and K.A. Ross, Abstract Harmonic Analysis I, Springer, 1963.

\bibitem{aH-Fell} A.~an Huef, \emph{The transformation groups whose $C^*$-algebras are Fell algebras}, Bull. London Math. Soc. \textbf{33} (2001), 73--76.

\bibitem{aH} A.~an Huef, \emph{Integrable actions and the transformation groups whose $C^*$-algebras have bounded trace}, Indiana Univ. Math. J. \textbf{51} (2002), 1197--1233.

\bibitem{aHKS} A.~an Huef, A. Kumjian and A. Sims, \emph{A Dixmier-Douady theorem for Fell algebras}, J. Funct. Anal. \textbf{260} (2011), 1543-1581.

%
\bibitem{aHRW} A. an Huef, I. Raeburn, Dana P. Williams, \emph{Properties preserved under Morita equivalences of $C^*$-algebras}, Proc. Amer. Math. Soc. \textbf{135} (2007), 1495--1503.

\bibitem{James} G.D. James, The representation theory of the symmetric groups, Lecture Notes in Mathematics, 682, Springer, Berlin, 1978.

\bibitem{KST} E. Kaniuth, G. Schlichting and K. F. Taylor, \emph{Minimal primal and Glimm ideal spaces of group $C^*$-algebras}, J. Funct. Anal. \textbf{130} (1995), 43--76.

\bibitem{KT} E. Kaniuth and K.T. Taylor, Induced Representations of Locally Compact Groups, Cambridge University Press, 2013.

    \bibitem{Knapp} A.W. Knapp, \emph{Branching theorems for compact symmetric spaces}, Represent. Theory, \textbf{5} (2001), 404--436.



\bibitem{MR} D. Marelli and I. Raeburn, \emph{Proper actions which are not saturated},  Proc. Amer. Math. Soc. \textbf{137} (2009), 2273Ð-2283.

\bibitem{Mont} D. Montgomery, \emph{Orbits of highest dimension}, in: Seminar on Transformation Groups, ed. A. Borel, Ann. of Math. Stud. 46, Chapter IX, Princeton Univ. Press, Princeton, NJ, 1960, 117--131.

\bibitem{Mur} F.D. Murnaghan, The Theory of Group Representations, Johns Hopkins Press, Baltimore, 1938.

\bibitem{Neu} K. Neumann, \emph{A description of the Jacobson topology on the spectrum of transformation group $C^*$-algebras by proper actions}, PhD thesis, University of M\"unster, 2011.

\bibitem{palais} R.S. Palais, \emph{On the existence of slices for actions of non-compact Lie groups}, Ann. of Math. \textbf{73} (1961), 295--323.

\bibitem{Ped} G.K. Pedersen, $C^*$-Algebras and their
    Automorphism Groups, Academic Press, London, 1979.

\bibitem{rae-sit} I. Raeburn, \emph{Induced $C^*$-algebras and a symmetric imprimitivity theorem},  Math. Ann. 280 (1988),  369--387.

%
\bibitem{tfb}
I. Raeburn and D.P. Williams, Morita Equivalence and
Continuous-Trace $C^*$-Algebras, Math. Surveys and Monographs, vol. {\bf
60}, Amer. Math. Soc.,  1998.

\bibitem{Rieffel}
M.A. Rieffel, \emph{Proper actions of groups on {$C^*$}-algebras},
Mappings of Operator Algebras (H. Araki and R.V. Kadison, eds.), Progress in Math.,
vol.~84, Birkhauser, 1988, pages~141--182.

\bibitem{rieffel2} M.A. Rieffel, \emph{Integrable and proper actions on $C^*$-algebras, and square-integrable representations of groups}, Expo. Math. \textbf{22} (2004), 1--53.


\bibitem{R} J. Rosenberg, \emph{Appendix to O. Bratteli's paper on ``Crossed products of UHF algebras''},  Duke Math. J.  \textbf{46} (1979), 25--26.



\bibitem{W2}
 D.P. Williams, \emph{The topology on the primitive ideal space of
 transformation group $C^*$-algebras and CCR transformation group
 $C^*$-algebras}, Trans. Amer. Math. Soc. \textbf{266} (1981), 335--359.


\bibitem{tfb^2}  D.P. Williams, Crossed Products of $C^*$-Algebras,
  Math. Surveys and Monographs, vol.~134, Amer. Math. Soc.,
   2007.
%

\end{thebibliography}
\end{document}